\documentclass[11pt,fleqn,mathserif]{article}
\usepackage{stmaryrd}
\usepackage{times}
\usepackage[top=2.0cm,bottom=2.5cm,left=2.5cm,right=2.5cm]{geometry}
\usepackage{amsmath,amssymb,mathrsfs,amsthm}

\usepackage[linkcolor=blue,pdfstartview=FitH,
            CJKbookmarks=true,bookmarksnumbered=true,bookmarksopen=true,
            colorlinks=true,
            pdfborder=000,
            citecolor=blue]{hyperref}
\usepackage{graphicx,graphics,subfigure,float,caption2,booktabs,epstopdf}
\usepackage{varioref,algorithm}
\labelformat{equation}{\textup{(#1)}}
\newtheorem{theorem}{Theorem}
\newtheorem{lemma}{Lemma}
\newtheorem{remark}{Remark}
\newtheorem{example}{Example}

\newtheorem{proposition}{Proposition}
\newtheorem{definition}{Definition}

\numberwithin{equation}{section}
\begin{document}
\title{Compact Finite Difference Approximations for Space Fractional Diffusion Equations}
\author{Han Zhou,~~WenYi Tian,~~Weihua Deng\footnote{Corresponding Author. E-mail: dengwh@lzu.edu.cn}\\[10pt]
    \itshape{School of Mathematics and Statistics, Lanzhou University,
             Lanzhou 730000, P R China}
    }
\date{}
\maketitle
\begin{abstract}
Based on the weighted and shifted Gr\"{u}nwald difference (WSGD)
operators \cite{Tian:12}, we further construct the compact finite
difference discretizations for the fractional operators. Then the
discretization schemes are used to approximate the one and two
dimensional space fractional diffusion equations. The detailed
numerical stability and error analysis are theoretically performed.
We theoretically prove and numerically verify that the provided
numerical schemes have the convergent orders $3$ in space and $2$ in
time.


\end{abstract}
\section{Introduction}
Diffusion is a fundamental phenomena in real world. The classical
diffusion equation can be derived from the conservation law of
particles, when assuming the particles' diffusion satisfies the
Fick's law. Based on the CTRW model of statistical physics, it can
also be easily derived. Nowadays, anomalous diffusion is widely
recognized in scientific community, its basic feature is that the
classical Fick's law doesn't not hold again. In fact, anomalous
diffusion is usually characterized by the mean square displacement
of the particles, given by $\langle  x^2(t) \rangle - \langle x(t)
\rangle^2 \sim t^\alpha$. The exponent $\alpha$ classifies the type
of diffusion: for $\alpha=1$, we have normal diffusion; for
$0<\alpha<1$, subdiffusion; and $\alpha>1$, superdiffusion. The
anomalous diffusion equations can still be derived from two ways:
using the conservation law and fractional Fick's law; and from the
CTRW model with power law waiting time and/or power law jump length
distribution. Fractional calculus plays an important role in
obtaining the anomalous diffusion equations \cite{Barkai:02,
Benson:00, Chaves:98, Gorenflo:98, Metzler:00, Metzler:04,
Scalas:00, Zaslavsky:02}.

The equation describing the superdiffusion is the space fractional
diffusion equation, and its concrete form is to replace the second
order derivative of the classical diffusion equation by the
Riemann-Liouville fractional derivative of order $\alpha$ and
$1<\alpha<2$ \cite{Metzler:00, Metzler:04}. This paper concerns the
high accurate numerical methods for the space fractional diffusion
equation. Meerschaert and his collaborators first introduce the
shifted Gr\"{u}nwald formula and successfully get the stable finite
difference scheme for numerically solving the space fractional
equation based on this formula \cite{Meerschaert:04}. Later more
research works are appeared by using the shifted Gr\"{u}nwald
formula to discretize the space fractional derivatives  \cite{Li:11,
Meerschaert:04, Meerschaert:06, Meerschaert:06b} with first order
accuracy in space, and possibly obtaining second order accuracy
after extrapolation \cite{Tadjeran:06, Tadjeran:07}. In
\cite{Tian:12}, the weighted and shifted Gr\"{u}nwald difference
(WSGD) operators are introduced to discretize the fractional
operators with second or higher order accuracy, then the second
order finite difference schemes are established to numerically solve
the space fractional diffusion equation. In \cite{Tian:12}, it is
also verified that the 3-WSGD operator can approximate the
Riemann-Liouville fractional derivative with third order accuracy,
but the 3-WSGD operator fails to numerically solve the time
dependent space fractional diffusion equations with unconditional
stability.

As the sequel to \cite{Tian:12}, based on the WSGD operators we
 focus on constructing the compact finite difference
discretizations, termed compact WSGD operators (CWSGD), for the
fractional operators with third order accuracy. When the order of
fractional derivative $\alpha$ equals to $1$ or $2$, it becomes the
compact difference operators for the first or second order spatial
derivatives with fourth order accuracy, which have been widely used
in numerically solving the linear and nonlinear, steady and
evolution equations to achieve the high order accuracy
\cite{Liao:08, Qin:11, Sengupta:04, Tian:07}. By using the CWSGD
operator, we get the third order numerical schemes for the space
fractional diffusion equation. We theoretically make the numerical
stability and error analysis, and the detailedly numerical
experiments confirm the theoretical results. In performing the
theoretical analysis, the negative definite property of the matrix
generated by the WSGD discretization to the fractional diffusion
operator \cite{Tian:12} still plays a key role.

This paper is organized as follows. In Sec. 2, we introduce the
CWSGD operators. Based on the CWSGD operators, in Sec. 3 and 4, we
build the compact finite difference schemes for the one and two
dimensional space fractional diffusion equations. The stability and
third order convergence with respect to the discrete $L^2$ norm of
the difference schemes are proven. In Sec. 5, the extensive
numerical experiments are performed to verify the accuracy and
theoretical convergent order. And some concluding remarks are made
in the last section.
\section{Compact WSGD Operators for the Riemann-Liouville
Fractional Derivatives} Now from the WSGD operator and the Taylor's
expansions of the shifted Gr\"unwald finite difference formulae,  we
derive the CWSGD operators for the Riemann-Liouville derivatives.
First let us introduce the definition of the Riemann-Liouville
derivatives.

\begin{definition}[\cite{Podlubny:99}]\label{def:1}
The $\alpha\,(n-1<\alpha<n)$ order left and right Riemann-Liouville
fractional derivatives of the function $u(x)$ on $(a,b)$ are,
respectively, defined as
\begin{itemize}
  \item[(1)] left Riemann-Liouville fractional derivative:
    \begin{equation*}
      _aD_x^{\alpha}u(x)=\frac{1}{\Gamma(n-\alpha)}\frac{\mathrm{d}^n}{\mathrm{d}
          x^n}\int_a^x\frac{u(\xi)}{(x-\xi)^{\alpha-n+1}}\mathrm{d}\xi;
    \end{equation*}
  \item[(2)] right Riemann-Liouville fractional derivative:
    \begin{equation*}
      _xD_b^{\alpha}u(x)=\frac{(-1)^n}{\Gamma(n-\alpha)}\frac{\mathrm{d}^n}{\mathrm{d}
          x^n}\int_x^b\frac{u(\xi)}{(\xi-x)^{\alpha-n+1}}\mathrm{d}\xi;
    \end{equation*}
\end{itemize}
 ~~ if $\alpha=n$, then $_aD_x^{\alpha}u(x)=\frac{\mathrm{d}^n}{\mathrm{d}x^n}u(x)$ and $_xD_b^{\alpha}u(x)=(-1)^n\frac{\mathrm{d}^n}{\mathrm{d}x^n}u(x)$.
\end{definition}
The second order finite difference approximations (the WSGD
operators) for Riemann-Liouville fractional derivatives are given as
follows.
\begin{lemma}[\cite{Tian:12}]
\label{lem:1} Let $u\in L^1(\mathbb{R})$,
$_{-\infty}D_x^{\alpha+2}u$, $_{x}D_{\infty}^{\alpha+2}u$ and their
Fourier transformations belong to $L^1(\mathbb{R})$, and define the
weighted and shifted Gr\"unwald difference (WSGD) operators by
 \begin{subequations}
 \begin{align}
    &{_L}\mathcal{D}_{h,p,q}^{\alpha}u(x)=\frac{1}{h^{\alpha}}\sum_{k=0}^{\infty}g_k^{(\alpha)}\Big(\frac{\alpha-2q}{2(p-q)}u(x-(k-p)h)+\frac{2p-\alpha}{2(p-q)}u(x-(k-q)h)\Big),\label{eq:2.1a} \\
    &{_R}\mathcal{D}_{h,p,q}^{\alpha}u(x)=\frac{1}{h^{\alpha}}\sum_{k=0}^{\infty}g_k^{(\alpha)}\Big(\frac{\alpha-2q}{2(p-q)}u(x+(k-p)h)+\frac{2p-\alpha}{2(p-q)}u(x+(k-q)h)\Big),\label{eq:2.1b}
  \end{align}
  \end{subequations}
  where $g_k^{(\alpha)}=(-1)^k\binom{\alpha}{k}$ are the coefficients of the power series of the function $(1-z)^{\alpha}$, and $p,q$ are integers, $p\neq q$.
  Then we have
  \begin{subequations}
  \begin{align}
    &{_L}\mathcal{D}_{h,p,q}^{\alpha}u(x)={_{-\infty}}D_x^{\alpha}u(x)+O(h^2),   \\
    &{_R}\mathcal{D}_{h,p,q}^{\alpha}u(x)={_x}D_{\infty}^{\alpha}u(x)+O(h^2),
  \end{align}
  \end{subequations}
  uniformly for $x\in\mathbb{R}$.
\end{lemma}
From \cite{Tadjeran:06}, we know that the Taylor's expansions of the
shifted Gr\"unwald finite difference formulae, which is necessary
for establishing the CWSGD operator for Riemann-Liouville fractional
derivatives.
\begin{lemma}[\cite{Tadjeran:06}]\label{lem:2}
  Assuming that $u\in C^{n+3}(\mathbb{R})$ such that all derivatives of $u$ up
  to order $n+3$ belong to $L^1(\mathbb{R})$,  then for any nonnegative integer $p$, we can obtain that
  \begin{subequations}
    \begin{align}
      & \frac{1}{h^{\alpha}}\sum_{k=0}^{\infty}g_k^{(\alpha)}u(x-(k-p)h)
        ={_{-\infty}}D_x^{\alpha}u(x)+\sum_{l=1}^{n-1}
        \Big(a^{\alpha}_{p,l}~{_{-\infty}}D_x^{\alpha+l}u(x)\Big)h^l+O(h^n),\label{eq:2.3a}       \\
      & \frac{1}{h^{\alpha}}\sum_{k=0}^{\infty}g_k^{(\alpha)}u(x+(k-p)h)
        ={_x}D_{\infty}^{\alpha}u(x)+\sum_{l=1}^{n-1}
        \Big(a^{\alpha}_{p,l}~{_x}D_{\infty}^{\alpha+l}u(x)\Big)h^l+O(h^n)\label{eq:2.3b},
    \end{align}
  \end{subequations}
  uniformly for $x\in \mathbb{R}$, where $a^{\alpha}_{p, l}$ are the coefficients of the power  series expansion of function $w_{\alpha, p}(z)=\Big(\frac{1-e^{-z}}{z}\Big)^\alpha e^{pz}$, and $
  w_{\alpha, p}(z)=\sum\limits_{k=0}^{\infty}a^{\alpha}_{p,k}z^k=1+(p-\frac{\alpha}{2})z+\frac{1}{24} (\alpha + 3 \alpha^2 - 12 \alpha p + 12 p^2)z^2+O(z^3)$.
\end{lemma}
For any function $u(x)$, denoting by $\delta_x^2$ the second order
central difference operator, that is $\delta_x^2
u(x)=\big(u(x-h)-2u(x)+u(x+h)\big)/h^2$, we introduce the following
finite difference operator
\begin{equation}\label{eq:2.5}
  \mathcal{C}_x u=(1+c_{p,q,2}^{\alpha} h^2\delta_x^2) u,
\end{equation}
where
$c_{p,q,2}^{\alpha}=\frac{\alpha-2q}{2(p-q)}a_{p,2}^{\alpha}+\frac{2p-\alpha}{2(p-q)}a_{q,2}^{\alpha}$.
We call $\mathcal{C}_x$ the CWSGD operator of Riemann-Liouville
fractional derivatives, of which the detailed construction is
described in the following proposition.
\begin{proposition}\label{prop:1}
  Under the conditions of Lemmas \ref{lem:1} and \ref{lem:2}, there
  exist
  \begin{equation}
     \begin{split}\label{eq:2.11}
       & {_L}\mathcal{D}_{h,p,q}^{\alpha}u(x)
         =\mathcal{C}_x~\Big({_{-\infty}}D_x^{\alpha}u(x)\Big)
         +c_{p,q,3}^{\alpha}~{_{-\infty}}D_x^{\alpha+3}u(x)h^3+O(h^4),\\
       & {_R}\mathcal{D}_{h,p,q}^{\alpha}u(x)
         =\mathcal{C}_x~\Big({_x}D_{\infty}^{\alpha}u(x)\Big)
         +c_{p,q,3}^{\alpha}~{_x}D_{\infty}^{\alpha+3}u(x)h^3+O(h^4),
     \end{split}
  \end{equation}
  uniformly for $x\in\mathbb{R}$, where $p,q$ are nonnegative integers and $p\neq q$. The operator $\mathcal{C}_x$ is defined in \ref{eq:2.5}.
\end{proposition}
\begin{proof}
  Substituting formulae \ref{eq:2.3a} and \ref{eq:2.3b} into \ref{eq:2.1a} and \ref{eq:2.1b}, respectively, and taking $n=4$, we arrive at
  \begin{equation}\label{eq:2.4}
    \begin{split}
      & {_L}\mathcal{D}_{h,p,q}^{\alpha}u(x)=\Big(1+c_{p,q,2}^{\alpha} h^2
        \frac{\mathrm{d}^2}{ \mathrm{d}x^2}\Big)\big({_{-\infty}}D_x^{\alpha}u(x)\big)
        +c_{p,q,3}^{\alpha}~{_{-\infty}}D_x^{\alpha+3}u(x)h^3+O(h^4), \\
      & {_R}\mathcal{D}_{h,p,q}^{\alpha}u(x)=\Big(1+c_{p,q,2}^{\alpha} h^2
        \frac{\mathrm{d}^2}{\mathrm{d}x^2}\Big)\big({_x}D_{\infty}^{\alpha}u(x)\big)
        +c_{p,q,3}^{\alpha}~{_x}D_{\infty}^{\alpha+3}u(x)h^3+O(h^4).
    \end{split}
  \end{equation}
  Since $ \delta_x^2 u=\frac{\mathrm{d}^2}{\mathrm{d} x^2} u+O(h^2)$, it yields
  \begin{equation}\label{eq:2.7}
    \mathcal{C}_x u=(1+c_{p,q,2}^{\alpha} h^2\frac{\mathrm{d}^2}{ \mathrm{d}x^2})u+O(h^4).
  \end{equation}
  Then we obtain the needed results by substituting $\ref{eq:2.7}$ into $\ref{eq:2.4}$.
\end{proof}
\begin{remark}
  If the function $u(x)$ is defined on the bounded interval $[a, b]$ with boundary condition $u(a)=0$ or $u(b)=0$, then the WSGD formulae approximating the $\alpha$ order left and right Riemann-Liouville fractional derivatives of $u(x)$ at each point x are written as
  \begin{equation}\label{eq:2.2}
    \begin{split}
      & {_L}\mathcal{D}_{h,p,q}^{\alpha}u(x)=\frac{\mu_1}{h^\alpha}
        \sum_{k=0}^{[\frac{x-a}{h}]+p}g_k^{(\alpha)}u(x-(k-p)h)
        +\frac{\mu_2}{h^\alpha}\sum_{k=0}^{[\frac{x-a}{h}]+q}g_k^{(\alpha)}u(x-(k-q)h), \\
      & {_R}\mathcal{D}_{h,p,q}^{\alpha}u(x)=\frac{\mu_1}{h^\alpha}
        \sum_{k=0}^{[\frac{b-x}{h}]+p}g_k^{(\alpha)}u(x+(k-p)h)
        +\frac{\mu_2}{h^\alpha}\sum_{k=0}^{[\frac{b-x}{h}]+q}g_k^{(\alpha)}u(x+(k-q)h),
    \end{split}
  \end{equation}
  where
  $\mu_1=\frac{\alpha-2q}{2(p-q)},~\mu_2=\frac{2p-\alpha}{2(p-q)}$.  After applying Proposition \ref{prop:1}, we have
  \begin{equation}
    \begin{split}\label{eq:2.12}
      & {_L}\mathcal{D}_{h,p,q}^{\alpha}u(x)
        =\mathcal{C}_x~\Big({_{a}}D_x^{\alpha}u(x)\Big)
        +c_{p,q,3}^{\alpha}~{_a}D_x^{\alpha+3}u(x)h^3+O(h^4), \\
      & {_R}\mathcal{D}_{h,p,q}^{\alpha}u(x)
        =\mathcal{C}_x~\Big({_x}D_{b}^{\alpha}u(x)\Big)
        +c_{p,q,3}^{\alpha}~{_x}D_{b}^{\alpha+3}u(x)h^3+O(h^4).
    \end{split}
  \end{equation}
\end{remark}

\begin{remark}\label{rem:2}
  When employing the finite difference method with WSGD formulae for numerically solving non-periodic fractional differential equations on bounded interval,
  $p,q$ should be chosen satisfying $|p|\le1,|q|\le1$ to ensure that the nodes at which the values of $u$ needed in \ref{eq:2.2} are within the bounded interval; otherwise, we need to use another way to discretize the fractional derivative when $x$ is close to the
  right/left boundary. It was indicated in \cite{Tian:12} that the approximation by formula \ref{eq:2.2} with $(p,q)=(0,-1)$ turns out to be unstable for time dependent problems. Then two sets of $(p,q)=(1,0),~(1,-1)$ can be selected to establish the difference scheme for fractional diffusion equations, which is also appropriate for the compact difference approximations \ref{eq:2.11}.
  The coefficients $c_{p, q, l}^{\alpha}$ in \ref{eq:2.5} with $(p,q)=(1,0),~(1,-1)$ are
  \begin{equation}\label{eq:2.6}
    \begin{cases}
      c_{1,0,2}^{\alpha}=\frac{1}{24}(7\alpha-3\alpha^2),~~ c_{1,0,3}^{\alpha}=\frac{1}{24}(\alpha^3-3\alpha^2+2\alpha), & \\
      c_{1,-1,2}^{\alpha}=\frac{1}{24}(\alpha-3\alpha^2+12),~~ c_{1,-1,3}^{\alpha}=\frac{1}{24}(\alpha^3-4\alpha).&
    \end{cases}
  \end{equation}
  For $\alpha=2$, the WSGD operators \ref{eq:2.2} are the centered difference approximation of second order derivative when $(p,q)$ equals to $(1,0)$ or $(1,-1)$, and the approximations \ref{eq:2.11} behave as the compact difference operators of second order derivative as $c_{1,0,2}^2=c_{1,-1,2}^2=\frac{1}{12}$ and $c_{1,0,3}^2=c_{1,-1,3}^2=0$;
  for $\alpha=1$ and $(p,q)=(1,0)$, $c_{1,0,2}^1=\frac{1}{6}$ and $c_{1,0,3}^1=0$, then the centered and compact difference schemes for first order derivative are recovered.
\end{remark}
  For the cases of $(p, q)=(1, 0)$ and $(p, q)=(1, -1)$, the WSGD schemes at every point $x_i=a+i\ h~ (h=\frac{b-a}{N},~1\le i\le N-1)$ are denoted as
  \begin{equation}
    \begin{split}\label{eq:2.10}
      & {_L}\mathcal{D}_{h,p,q}^{\alpha}u(x_i)
        =\frac{1}{h^\alpha}\sum_{k=0}^{i+1}w_k^{(\alpha)}u(x_{i-k+1}),  \\
      & {_R}\mathcal{D}_{h,p,q}^{\alpha}u(x_i)
        =\frac{1}{h^\alpha}\sum_{k=0}^{N-i+1}w_k^{(\alpha)}u(x_{i+k-1}),
    \end{split}
  \end{equation}
  where
  \begin{equation}\left\{
    \begin{split}\label{eq:2.8}
      (p,q)=(1,0), \quad & w_0^{(\alpha)}=\frac{\alpha}{2}g_0^{(\alpha)},~
       w_k^{(\alpha)}=\frac{\alpha}{2}g_k^{(\alpha)}+\frac{2-\alpha}{2}g_{k-1}^{(\alpha)},~k\ge1;\\
      (p,q)=(1,-1), \quad & w_0^{(\alpha)}=\frac{2+\alpha}{4}g_0^{(\alpha)},~w_1^{(\alpha)}
       =\frac{2+\alpha}{4}g_1^{(\alpha)},\\
      & w_k^{(\alpha)}=\frac{2+\alpha}{4}g_k^{(\alpha)}
        +\frac{2-\alpha}{4}g_{k-2}^{(\alpha)},~k\ge2.
    \end{split}\right.
  \end{equation}
\begin{remark}\label{rem:3}
  Let $S_{N-1}$ be a symmetric tri-diagonal matrix of $(N-1)$-square, denoted by  $\mathrm{tridiag} (1, -2, 1)$. And we have the eigenvalues of the matrix $S_{N-1}$ in decreasing order (see \cite{Leveque:07}),
  \begin{equation}
    \lambda_k (S_{N-1})=-4\sin^2\Big(\frac{k\pi }{2N}\Big), ~~~ k=1,2, \cdots, N-1.
  \end{equation}
  Define
  \begin{equation}\label{eq:2.20}
    C_\alpha=I_{N-1}+c_{p,q,2}^{\alpha}S_{N-1},
  \end{equation}
  where $I_{N-1}$ is the unit matrix of $(N-1)$-square. Then the eigenvalues of $C_\alpha$ are given by
  \begin{equation}
    \lambda_k(C_\alpha)=1-4~c_{p, q,2}^{\alpha}\sin^2\Big(\frac{k\pi }{2N}\Big), ~~~
    k=1,2,\cdots,N-1.
  \end{equation}
  For the case of $(p, q)=(1, 0)$, we have
  \begin{equation}\label{eq:2.16}
    \lambda_k(C_\alpha) > 1-4~c_{1, 0, 2}^{\alpha}\ge
    \frac{23}{72}>0,
  \end{equation}
when $0<\alpha<\frac{7}{3}$; and $\lambda_k(C_\alpha) \ge 1$ when
$\alpha \ge \frac{7}{3}$.

 Then, from the Rayleigh-Ritz Theorem (see Theorem 8.8 in \cite{Zhang:11}), we know that the matrix $C_\alpha=(I_{N-1}+c_{1,0,2}^{\alpha}S_{N-1})$ is positive definite. And for the case of $(p, q)=(1, -1)$, we have
  \begin{equation}\label{eq:2.17}
    \lambda_k(C_\alpha) > 1-4~c_{1, -1, 2}^{\alpha} > 0 ~~~\text{iff}~~~\frac{1+\sqrt{73}}{6}<\alpha<\frac{1+\sqrt{145}}{6},
  \end{equation}
 and $\lambda_k(C_\alpha)\ge1$ when $\alpha\ge\frac{1+\sqrt{145}}{6}$; and $1-4~c_{1, -1, 2}^{\alpha}=0$ when $\alpha=\frac{1+\sqrt{73}}{6}$.  Thus, the matrix $C_\alpha=(I_{N-1}+c_{1,-1,2}^{\alpha}S_{N-1})$ is positive definite for any natural number $N$ if and only if  $\alpha\in [\frac{1+\sqrt{73}}{6}, \infty)$.
\end{remark}
\begin{lemma}[\cite{Tian:12}]\label{thm:4}
  Let the matrix $A_{\alpha}$ be of the following form
  \begin{equation}\label{eq:2.9}
    A_{\alpha}=
    \begin{pmatrix}
      w_1^{(\alpha)}     & w_0^{(\alpha)}     &                &                &  \\
      w_2^{(\alpha)}     & w_1^{(\alpha)}     & w_0^{(\alpha)} &                &  \\
      \vdots             & w_2^{(\alpha)}     & w_1^{(\alpha)} & \ddots         &  \\
      w_{N-2}^{(\alpha)} & \cdots             & \ddots         & \ddots         & w_0^{(\alpha)} \\
      w_{N-1}^{(\alpha)} & w_{N-2}^{(\alpha)} & \cdots         & w_2^{(\alpha)} & w_1^{(\alpha)} \\
    \end{pmatrix},
  \end{equation}
  where the diagonals $\{w_k^{(\alpha)}\}_{k=0}^{N-1}$ are the coefficients given in \ref{eq:2.8} corresponding to $(p,q)=(1,0),~(1,-1)$. Then we have that any eigenvalue $\lambda$ of $A_{\alpha}$ satisfies
  \begin{itemize}
    \item[(1)] $\mathrm{Re}(\lambda)\equiv0$, for $(p,q)=(1,0)$, $\alpha=1$,
    \item[(2)] $\mathrm{Re}(\lambda)<0$, for $(p,q)=(1,0)$, $1<\alpha\le2$,
    \item[(3)] $\mathrm{Re}(\lambda)<0$, for $(p,q)=(1,-1)$, $1\le\alpha\le2$.
  \end{itemize}
  Moreover, when $1<\alpha\le2$, the matrix $A_{\alpha}$ is negative definite, and the real parts of the eigenvalues $\lambda$ of matrix $c_1A_{\alpha}+c_2A_{\alpha}^{\mathrm{T}}$ are less than 0, where $c_1,c_2\ge0,c_1^2+c_2^2\neq0$.
\end{lemma}
\begin{lemma}[\cite{Minc:64, Laub:05}]\label{thm:2}
  The Matrix $A\in \mathbb{R}^{n\times n}$ is asymptotically stable if and only if
  there exists a symmetric and positive (or negative) definite solution $X\in \mathbb{R}^{n\times n}$ to the Lyapunov equation
  \begin{equation}
    AX+XA^{\mathrm{T}}=C,
  \end{equation}
  where $C=C^\mathrm{T}\in \mathbb{R}^{n\times n}$ is a negative (or positive) definite matrix. And a matrix $A$ is called asymptotically stable if all its eigenvalues have real parts in the open left half-plane, i.e., $Re\lambda (A)<0$.
\end{lemma}
Lemma \ref{thm:4} and \ref{thm:2} play a central role in analyzing
the stability and convergence of the compact difference
approximations in the sequel. Next taking $(p,q)=(1,0)$ and
$(p,q)=(1,-1)$, respectively, we consider the compact difference
schemes of approximating the space fractional diffusion equations.
\section{One Dimensional Space Fractional Diffusion Equation}
In this section, we consider the following two-sided one dimensional
space fractional diffusion equation
\begin{equation}\label{eq:3.1}
  \begin{cases}
   \displaystyle \frac{\partial u(x, t)}{\partial t}=K_1~{_a}D_x^{\alpha}u(x, t)+K_2~{_x}D_b^{\alpha}u(x, t)+f(x,t),  &   \text{$(x, t) \in (a, b)\times (0, T]$,} \\
    u(x,0)=u_0(x),    & \text{$x\in [a, b]$},\\
    u(a,t)=\phi_a(t),\ \ u(b,t)=\phi_b(t), &\text{$t\in [0, T]$},
  \end{cases}
\end{equation}
where ${_a}D_x^{\alpha}$ and $_xD_b^{\alpha}$ are the left and right
Riemann-Liouville fractional derivatives with $1<\alpha\leq 2$,
respectively. The diffusion coefficients $K_1$ and $K_2$ are
nonnegative constants with $K_1^2+K_2^2\neq0$. If $K_1\neq0$, then
$\phi_a(t)\equiv0$, and if $K_2\neq0$, then $\phi_b(t)\equiv0$. In
the following analysis of the numerical method, we assume that
\ref{eq:3.1} has a unique and sufficiently smooth solution.

\subsection{Compact Difference Scheme}
We partition the interval $[a, b]$ into a uniform mesh with the
space stepsize $h=(b-a)/N$ and the time stepsize $\tau=T/M$, where
$N, M$ are two positive integers. And the set of grid points are
denoted by $x_i=a+ih$ and $t_n=n\tau$ for $0\leq i\leq N$ and $0\leq
n\leq M$. Denoting $t_{n+1/2}=(t_n+t_{n+1})/2$ for $0\leq n\leq
M-1$, we introduce the following notations
\begin{equation*}
  u_i^n=u(x_i, t_n), \ \ u_i^{n+1/2}=\frac{1}{2}(u(x_i, t_n)+u(x_i, t_{n+1})), \ \ f_i^{n+1/2}=f(x_i, t_{n+1/2}), \ \ \delta_t u_i^n=(u_i^{n+1}-u_i^n)/\tau.
\end{equation*}
Employing the Crank-Nicolson technique for time discretization of \ref{eq:3.1}, we get
\begin{equation}\label{eq:3.2}
  \delta_t u_i^n-\frac{1}{2}\Big(K_1({_a}D_x^{\alpha}u)_i^n+K_1({_a}D_x^{\alpha}u)_i^{n+1}
  +K_2({_x}D_b^{\alpha}u)_i^n+K_2({_x}D_b^{\alpha}u)_i^{n+1} \Big)=f_i^{n+1/2}+O(\tau^2).
\end{equation}
Recalling the definition of operator $\mathcal{C}_x$ and \ref{eq:2.12}, we have
\begin{equation}\label{eq:3.4}
  \begin{split}
    & \mathcal{C}_x\Big({_a}D_x^{\alpha}u_i\Big)
      ={_L}\mathcal{D}_{h,p,q}^{\alpha}u_i
      -c_{p,q,3}^{\alpha}h^3~{_a}D_x^{\alpha+3}u_i+O(h^4), \\
   & \mathcal{C}_x\Big({_x}D_{b}^{\alpha}u_i\Big)
     ={_R}\mathcal{D}_{h,p,q}^{\alpha}u_i
     -c_{p,q,3}^{\alpha}h^3~{_x}D_{b}^{\alpha+3}u_i+O(h^4),
  \end{split}
\end{equation}
where ${_L}\mathcal{D}_{h,p,q}^{\alpha}$ and
${_R}\mathcal{D}_{h,p,q}^{\alpha}$ are given in \ref{eq:2.10}.
Acting $\tau\mathcal{C}_x$ on both sides \ref{eq:3.2} and then
plugging \ref{eq:3.4} into it, we obtain
\begin{equation}
\label{eq:3.5}
  \begin{split}
    \mathcal{C}_x&u_i^{n+1}-\frac{K_1\tau }{2}{_L}\mathcal{D}_{h,p,q}^{\alpha}u_i^{n+1}
      -\frac{K_2\tau }{2}{_R}\mathcal{D}_{h,p,q}^{\alpha}u_i^{n+1}  \\
    &=\mathcal{C}_xu_i^n+\frac{K_1\tau}{2}{_L}\mathcal{D}_{h,p,q}^{\alpha}u_i^n+\frac{K_2\tau}{2}{_R}\mathcal{D}_{h,p,q}^{\alpha}u_i^n+\tau \mathcal{C}_xf_i^{n+1/2}+\tau\varepsilon_i^{n+1/2},
  \end{split}
\end{equation}
where
\begin{equation}\label{eq:3.3}
  \varepsilon_i^{n+1/2}=-\Big(K_1c_{p,q,3}^{\alpha}~{_a}D_x^{\alpha+3}u_i^{n+1/2}+K_2c_{p,q,3}^{\alpha}~
  {_x}D_b^{\alpha+3}u_i^{n+1/2}\Big)h^3+O(\tau^2+h^4).
\end{equation}
Denoting by $U_i^n$ the numerical approximation of $u_i^n$, we derive the compact difference scheme
\begin{equation}
  \begin{split}\label{eq:3.7}
    \mathcal{C}_x&U_i^{n+1}-\frac{K_1\tau}{2h^{\alpha}}\sum_{k=0}^{i+1}w_k^{(\alpha)}U_{i-k+1}^{n+1}
        -\frac{K_2\tau}{2h^{\alpha}}\sum_{k=0}^{N-i+1}w_k^{(\alpha)}U_{i+k-1}^{n+1} \\
    &=\mathcal{C}_xU_i^n+\frac{K_1\tau}{2h^{\alpha}}\sum_{k=0}^{i+1}w_k^{(\alpha)}U_{i-k+1}^n
        +\frac{K_2\tau}{2h^{\alpha}}\sum_{k=0}^{N-i+1}w_k^{(\alpha)}U_{i+k-1}^n+\tau \mathcal{C}_xf_i^{n+1/2}.
  \end{split}
\end{equation}
For the convenience of implementation, we define
\begin{equation*}
  U^n=\Big(U_1^n, U_2^n,\cdots, U_{N-1}^n\Big)^{\mathrm{T}}, \ \ \ F^n=\Big(f_1^{n+1/2}, f_2^{n+1/2}, \cdots,
  f_{N-1}^{n+1/2}\Big)^{\mathrm{T}},
\end{equation*}
and reformulate the scheme \ref{eq:3.7} as the following matrix form
\begin{equation}\label{eq:3.8}
  \Big(C_{\alpha}-\frac{\tau}{2h^{\alpha}}(K_1A_{\alpha}+K_2A_{\alpha}^{\mathrm{T}})\Big)U^{n+1}
  =\Big(C_{\alpha}+\frac{\tau}{2h^{\alpha}}(K_1A_{\alpha}+K_2A_{\alpha}^{\mathrm{T}})\Big)U^{n}+\tau C_{\alpha}F^n+H^n,
\end{equation}
where $A_{\alpha}$ and $C_{\alpha}$ are given in \ref{eq:2.9} and
\ref{eq:2.20}, respectively, and
\begin{equation}
\begin{split}
 H^n=& \begin{bmatrix}
   c_{p,q,2}^{\alpha}\\
   0\\
   \vdots \\
   0\\
   0
 \end{bmatrix}(U_0^n-U_0^{n+1}+\tau f_0^{n+1/2})+
  \begin{bmatrix}
   0\\
   0\\
   \vdots \\
   0\\
   c_{p,q,2}^{\alpha}
 \end{bmatrix}(U_N^n-U_N^{n+1}+\tau f_N^{n+1/2})+  \\
 &\frac{\tau}{2h^{\alpha}}
 \begin{bmatrix}
   K_1w_2^{(\alpha)}+K_2w_0^{(\alpha)}\\
   K_1w_3^{(\alpha)}\\
   \vdots \\
   K_1w_{N-1}^{(\alpha)}\\
   K_1w_N^{(\alpha)}
 \end{bmatrix}(U_{0}^n+U_{0}^{n+1})+
 \frac{\tau}{2h^{\alpha}}
 \begin{bmatrix}
   K_2w_N^{(\alpha)}\\
   K_2w_{N-1}^{(\alpha)}\\
   \vdots \\
   K_2w_3^{(\alpha)}\\
   K_1w_0^{(\alpha)}+K_2w_2^{(\alpha)}
 \end{bmatrix}(U_{N}^n+U_{N}^{n+1}).
 \end{split}
\end{equation}
\subsection{Stability and Convergence}
Next we consider the stability and convergence analysis for the scheme \ref{eq:3.7}. Let
\begin{equation*}
  V_h=\big\{v:v=\{v_{i}\} \text{ is a grid function in } \{x_i=a+ih\}_{i=0}^{N}\text{~and~}v_{0}=v_{N}=0\big\}.
\end{equation*}
For any $v=\{v_{i}\}\in V_h$, we define its pointwise maximum norm
and the discrete $L^2$ norm as
\begin{equation}
  \|v\|_{\infty}=\max\limits_{1\le i\le N-1} |v_{i}|,~~~\|v\|^2=h\sum_{i=1}^{N-1}v_{i}^2.
\end{equation}
\begin{theorem}
  For the case of $(p,q)=(1,0)$, the difference scheme \ref{eq:3.7} is unconditionally stable for all $1<\alpha\le 2$; and for the case of $(p,q)=(1,-1)$,  the difference scheme \ref{eq:3.7} is also unconditionally stable for $\frac{1+\sqrt{73}}{6}\le\alpha\le 2$.
\end{theorem}
\begin{proof}
  Denoting $D_{\alpha}=\frac{\tau}{2h^{\alpha}}(K_1A_{\alpha}+K_2A_{\alpha}^{\mathrm{T}})$, we rewrite \ref{eq:3.8} as
  \begin{equation}
    (C_{\alpha}-D_{\alpha})U^{n+1}=(C_{\alpha}+D_{\alpha})U^{n}+\tau C_{\alpha}F^n+H^n.
  \end{equation}
From Remark \ref{rem:3}, we know that $C_{\alpha}$ is a symmetric
and positive definite matrix when $(p, q)=(1, 0)$ with $1<\alpha\le
2$ and $(p, q)=(1, -1)$ with $\frac{1+\sqrt{73}}{6}\le\alpha\le 2$,
which follows that $C_{\alpha}^{-1}$ is also symmetric and positive
definite. On the other hand, Lemma \ref{thm:4} shows that the
eigenvalues of the matrix
$\frac{D_{\alpha}+D_{\alpha}^{\mathrm{T}}}{2}=\frac{\tau(K_1+K_2)}{4h^{\alpha}}(A_{\alpha}+A_{\alpha}^{\mathrm{T}})$
are all negative for $1<\alpha\le 2$, thus
$(D_{\alpha}+D_{\alpha}^{\mathrm{T}})$ is a symmetric and negative
definite matrix. Then, for any
$v=(v_1,v_2,\cdots,v_{N-1})^{\mathrm{T}}\in\mathbb{R}^{N-1}\setminus{0}$,
there exists
  \begin{equation}
    v^{\mathrm{T}}\Big((C_{\alpha}^{-1}D_{\alpha})C_{\alpha}^{-1}
    +C_{\alpha}^{-1}(C_{\alpha}^{-1}D_{\alpha})^{\mathrm{T}}\Big)v
    =v^{\mathrm{T}}C_{\alpha}^{-1}(D_{\alpha}+D_{\alpha}^{\mathrm{T}})C_{\alpha}^{-1}v < 0,
  \end{equation}
which means that the matrix
$\big((C_{\alpha}^{-1}D_{\alpha})C_{\alpha}^{-1}+C_{\alpha}^{-1}(C_{\alpha}^{-1}D_{\alpha})^{\mathrm{T}}\big)$
is negative definite. Then it yields from Lemma \ref{thm:2} that all
the eigenvalues of $({C_{\alpha}^{-1}D_{\alpha}})$ have negative
real parts. In addition, $\lambda$ is an eigenvalue of
$({C_{\alpha}^{-1}D_{\alpha}})$ if and only if
$\frac{1+\lambda}{1-\lambda}$ is an eigenvalue of matrix
$(I-C_{\alpha}^{-1}D_{\alpha})^{-1}(I+C_{\alpha}^{-1}D_{\alpha})$,
and $|\frac{1+\lambda}{1-\lambda}|<1$ holds. Hence, the spectral
radius of the matrix
$(C_{\alpha}-D_{\alpha})^{-1}(C_{\alpha}+D_{\alpha})=(I-C_{\alpha}^{-1}D_{\alpha})^{-1}(I+C_{\alpha}^{-1}D_{\alpha})$
is less than one, and the difference scheme \ref{eq:3.7} is stable.
\end{proof}
\begin{lemma}[Discrete Gronwall Lemma \cite{Quarteroni:97}]\label{lem:3}
  Assume that $\{k_n\}$ and $\{p_n\}$ are nonnegative sequences, and the sequence $\{\phi_n\}$ satisfies
  \begin{equation*}
    \phi_0\leq g_0,\ \ \ \ \phi_n\leq g_0+\sum_{l=0}^{n-1}p_l+\sum_{l=0}^{n-1}k_l\phi_l,\ \ \ n\geq 1,
  \end{equation*}
  where $g_0\geq 0$. Then the sequence $\{\phi_n\}$ satisfies
  \begin{equation}
    \phi_n\leq \Big(g_0+\sum_{l=0}^{n-1}p_l\Big)\exp\Big(\sum_{l=0}^{n-1}k_l\Big),\ \ \ n\geq 1.
  \end{equation}
\end{lemma}
\begin{theorem}
  Let $u_i^n$ be the exact solution of problem \ref{eq:3.1}, and $U_i^n$ be the solution of difference scheme \ref{eq:3.7} at grid point $(x_i, t_n)$. Then the following estimate
  \begin{equation}\label{eq:3.14}
    \|u^n-U^n\|\leq c(\tau^2+h^3), ~~~1\leq n\leq M,
  \end{equation}
  holds for all $1<\alpha< 2$ with $(p,q)=(1,0)$ and $\frac{1+\sqrt{73}}{6}<\alpha< 2$ with $(p,q)=(1,-1)$.
\end{theorem}
\begin{proof}
 Denoting $e_i^n=u_i^n-U_i^n$, from formulae \ref{eq:3.5} and \ref{eq:3.7} we have
  \begin{equation}\label{eq:3.15}
    C_{\alpha}(e^{n+1}-e^{n})-\frac{K_1\tau}{2h^{\alpha}}A_{\alpha}(e^{n+1}+e^n)
    -\frac{K_2\tau}{2h^{\alpha}}A_{\alpha}^{\mathrm{T}}(e^{n+1}+e^n)=\tau \varepsilon^{n+1/2},
  \end{equation}
  where
  \begin{equation*}
    e^n=(e_1^n, e_2^n, \cdots,e_{N-1}^n)^{\mathrm{T}},\quad \varepsilon^{n+1/2}=(\varepsilon_1^{n+1/2},\varepsilon_2^{n+1/2},\cdots,\varepsilon_{N-1}^{n+1/2})^{\mathrm{T}}.
  \end{equation*}
  Multiplying \ref{eq:3.15} by $h(e^{n+1}+e^n)^\mathrm{T}$, we obtain that
  \begin{equation}\label{eq:3.16a}
    \begin{split}
      & h(e^{n+1}+e^n)^{\mathrm{T}}C_{\alpha}(e^{n+1}-e^{n})
        -\frac{K_1\tau}{2h^{\alpha-1}}(e^{n+1}+e^n)^{\mathrm{T}}A_{\alpha}(e^{n+1}+e^n)    \\
      & -\frac{K_2\tau}{2h^{\alpha-1}}(e^{n+1}+e^n)^{\mathrm{T}}A_{\alpha}^{\mathrm{T}}
        (e^{n+1}+e^n)=\tau h(e^{n+1}+e^n)^{\mathrm{T}}\varepsilon^{n+1/2}.
    \end{split}
  \end{equation}
  By Lemma \ref{thm:4}, $A_{\alpha}$ and its transpose $A_{\alpha}^\mathrm{T}$ are both negative semi-definite matrices for $1\le \alpha\le2$, thus
  \begin{equation}
    (e^{n+1}+e^n)^{\mathrm{T}}A_{\alpha}(e^{n+1}+e^n)\le 0, \ \ \ \ (e^{n+1}+e^n)^{\mathrm{T}}A_{\alpha}^{\mathrm{T}}(e^{n+1}+e^n)\le 0.
  \end{equation}
  Then \ref{eq:3.16a} leads to
  \begin{equation}\label{eq:3.16b}
    h(e^{n+1}+e^n)^{\mathrm{T}}C_{\alpha}(e^{n+1}-e^{n})\le\tau h(e^{n+1}+e^n)^{\mathrm{T}}\varepsilon^{n+1/2}.
  \end{equation}
  As the matrix $C_{\alpha}$ is symmetric, we derive that
  \begin{equation}
    h(e^{n+1}+e^n)^{\mathrm{T}}C_{\alpha}(e^{n+1}-e^{n})=E^{n+1}-E^n,
  \end{equation}
  where
  \begin{equation}
    E^n=h(e^n)^{\mathrm{T}}C_{\alpha}(e^{n})\ge (1-4~c_{p,q,2}^{\alpha})\|e^n\|^2.
  \end{equation}
 From \ref{eq:2.16} and \ref{eq:2.17}, it yields $E^n\ge \lambda\|e^n\|^2$, where $\lambda=1-4c_{1,-1,2}^{\alpha}>0$ if $\frac{1+\sqrt{73}}{6}<\alpha\le 2$ and $(p,q)=(1,-1)$;  and $\lambda=\frac{23}{72}$ if $1\le \alpha\le 2$ and $(p,q)=(1,0)$. Together with \ref{eq:3.16b}, it yields that
  \begin{equation}
    E^{k+1}-E^k\le \tau h(e^{k+1}+e^k)^{\mathrm{T}}\varepsilon^{k+1/2}
    \le \frac{\tau\lambda}{2}\big(\|e^{k+1}\|^2+\|e^{k}\|^2\big)
    +\frac{\tau}{\lambda}\|\varepsilon^{k+1/2}\|^2.
  \end{equation}
  Summing up for all $0\leq k \leq n-1$, we have
  \begin{equation}
    \begin{split}
      \lambda\|e^n\|^2&\le \tau h(e^{n}+e^{n-1})^{\mathrm{T}}\varepsilon^{n-1/2} +\frac{\tau\lambda}{2}\sum_{k=0}^{n-2}\big(\|e^{k+1}\|^2+\|e^{k}\|^2\big)
      +\frac{\tau}{\lambda}\sum_{k=1}^{n-2}\|\varepsilon^{k+1/2}\|^2          \\
      &\le \frac{\lambda}{2}\|e^{n}\|^2+\frac{\tau^2}{2\lambda}\|\varepsilon^{n-1/2}\|^2
      +\tau\lambda\sum_{k=1}^{n-1}\|e^{k}\|^2
      +\frac{\tau}{\lambda}\sum_{k=1}^{n-1}\|\varepsilon^{k+1/2}\|^2.
    \end{split}
  \end{equation}
   Since $|\varepsilon_i^{k+1/2}|\le \tilde{c}(\tau^2+h^3)$ for any $0\leq k \leq n-1$, then it leads to
  \begin{equation}
    \begin{split}
      \|e^n\|^2&\le 2\tau\sum_{k=1}^{n-1}\|e^{k}\|^2
      +\frac{2\tau}{\lambda^2}\sum_{k=1}^{n-1}\|\varepsilon^{k+1/2}\|^2
      +\frac{\tau^2}{\lambda^2}\|\varepsilon^{n-1/2}\|^2 \\
      &\le 2\tau\sum_{k=1}^{n-1}\|e^{k}\|^2+c(\tau^2+h^3)^2,
    \end{split}
  \end{equation}
  which completes the proof by Lemma \ref{lem:3}.
\end{proof}
\begin{remark}
The truncation error in \ref{eq:3.3} becomes
$\varepsilon_i^{n+1/2}=O(\tau^2+h^4)$ when $\alpha=1,2$ with
$(p,q)=(1,0)$ and $\alpha=2$ with $(p,q)=(1,-1)$, so when taking
$\alpha=1,2$ the compact finite difference schemes for the classical
diffusion equations are recovered and the corresponding error
estimate of the difference scheme \ref{eq:3.7} satisfies
  \begin{equation}
    \|u^n-U^n\|\leq c(\tau^2+h^4), ~~~1\leq n\leq M.
  \end{equation}
\end{remark}
\section{Two Dimensional Space Fractional Diffusion Equation}
Now we consider the following two-sided space fractional diffusion
equation in two dimensions
\begin{equation}\label{eq:4.1}
  \begin{cases}
   \displaystyle \frac{\partial u(x, y, t)}{\partial t}=\Big(K_1^{+}{_a}D_x^{\alpha}u(x, y, t)+K_2^{+}{_x}D_b^{\alpha}u(x, y, t)\Big)\\
    ~~~~~~~~~~~~~~~~~+\Big(K_1^{-}{_c}D_y^{\beta}u(x, y, t)+K_2^{-}{_y}D_d^{\beta}u(x, y, t)\Big)+f(x, y, t),  &   \text{$(x, y, t) \in \Omega\times (0, T]$,} \\
    u(x, y, 0)=u_0(x, y),    & \text{$(x, y)\in \Omega$},\\
    u(x, y, t)=\varphi(x, y, t), &\text{$(x, y, t)\in\partial\Omega\times[0, T]$},
  \end{cases}
\end{equation}
where $\Omega=(a,b)\times(c,d)$, ${_a}D_x^{\alpha},
{_x}D_b^{\alpha}$ and ${_c}D_y^{\beta}, {_y}D_d^{\beta}$ are
Riemann-Liouville fractional derivatives with $1<\alpha, \beta \leq
2$. The diffusion coefficients $K_i^{+}, ~K_i^{-}\geq 0,~i=1,2,$
satisfy $(K_1^{+})^2+(K_2^{+})^2\neq0$ and
$(K_1^{-})^2+(K_2^{-})^2\neq0$. And $\varphi(a,y,t)\equiv0$ if
$K_1^{+}\neq0$; $\varphi(b,y,t)\equiv0$ if $K_1^{-}\neq0$;
$\varphi(x,c,t)\equiv0$ if $K_2^{+}\neq0$; $\varphi(x,d,t)\equiv0$
if $K_2^{-}\neq0$. We assume that \ref{eq:4.1} has a unique and
sufficiently smooth solution.
\subsection{Compact Difference Scheme}
 We partition the domain $\Omega$ into a
uniform mesh with the space stepsizes $h_x=(b-a)/N_x, h_y=(d-c)/N_y$
and the time stepsize $\tau=T/M$, where $N_x, N_y, M$ being positive
integers. And the set of grid points are denoted by $x_i=a+ih_x,
y_j=c+jh_y$ and $t_n=n\tau$ for $0\leq i\leq N_x, 0\leq j\leq N_y$
and $0\leq n\leq M$. Setting $t_{n+1/2}=(t_n+t_{n+1})/2$  for $0\leq
n\leq M-1$, we denote \begin{equation*}
 \begin{split}
  & u_{i, j}^n=u(x_i, y_j, t_n), \ \ u_{i,j}^{n+1/2}=\frac{1}{2}\big(u(x_i, y_j, t_n)+u(x_i, y_j, t_{n+1})\big),  \\ &  f_{i, j}^{n+1/2}=f(x_i, y_j, t_{n+1/2}), \ \ \delta_t u_{i, j}^n=(u_{i, j}^{n+1}-u_{i, j}^n)/\tau.
 \end{split}
\end{equation*}
Discretizing \ref{eq:4.1} by the Crank-Nicolson technique in time
direction leads to
\begin{equation}\label{eq:4.2}
  \begin{split}
    \delta_t u_{i,j}^{n}&=\frac{1}{2}\Big(K_1^{+}(_aD_x^{\alpha}u)_{i,j}^{n+1}+K_2^{+}(_xD_b^{\alpha}u)_{i,j}^{n+1}+K_1^{-}(_cD_y^{\beta}u)_{i,j}^{n+1}+K_2^{-}(_yD_d^{\beta}u)_{i,j}^{n+1}                           \\
    &+K_1^{+}(_aD_x^{\alpha}u)_{i,j}^n+K_2^{+}(_xD_b^{\alpha}u)_{i,j}^n+K_1^{-}(_cD_y^{\beta}u)_{i,j}^n+K_2^{-}(_yD_d^{\beta}u)_{i,j}^n\Big)+f_{i,j}^{n+1/2}+O(\tau^2).
  \end{split}
\end{equation}
In the space discretizations, we introduce the finite difference
operators
\begin{equation}
\mathcal{C}_xu_{i,j}=( 1+c_{p,q,2}^{\alpha}h_x^2\delta_x^2 )u_{i,j}, ~~~\mathcal{C}_yu_{i,j}=( 1+c_{p,q,2}^{\beta}h_y^2\delta_y^2 )u_{i,j},
\end{equation}
and deduce that
\begin{subequations}
\begin{align}
&\mathcal{C}_x\Big({_a}D_x^{\alpha}u_{i,j}\Big)={_L}\mathcal{D}_{h_x,p,q}^{\alpha}u_{i,j}-c_{p,q,3}^{\alpha}h_x^3~_a
D_x^{\alpha+3}u_{i,j}+O(h_x^4),  \label{eq:4.4a}\\
&\mathcal{C}_x\Big({_x}D_b^{\alpha}u_{i,j}\Big)={_R}\mathcal{D}_{h_x,p,q}^{\alpha}u_{i,j}-c_{p,q,3}^{\alpha}h_x^3~_x
D_b^{\alpha+3}u_{i,j}+O(h_x^4),  \label{eq:4.4b}\\
&\mathcal{C}_y\Big({_c}D_y^{\beta}u_{i,j}\Big)={_L}\mathcal{D}_{h_y,p,q}^{\beta}u_{i,j}-c_{p,q,3}^{\beta}h_y^3~
{_c}D_y^{\beta+3}u_{i,j}+O(h_y^4), \label{eq:4.4c}\\
&\mathcal{C}_y\Big({_y}D_d^{\beta}u_{i,j}\Big)={_R}\mathcal{D}_{h_y,p,q}^{\beta}u_{i,j}-c_{p,q,3}^{\beta}h_y^3~
{_y}D_d^{\beta+3}u_{i,j}+O(h_y^4).\label{eq:4.4d}
\end{align}
\end{subequations}
For the simplification of presentation, the same stepsizes are
chosen in the following discussions, and denoted as $h=h_x=h_y$.
Acting $\tau\mathcal{C}_x\mathcal{C}_{y}$ on both sides of
\ref{eq:4.2} and plugging \ref{eq:4.4a}-\ref{eq:4.4d} in it, we have
\begin{equation}\label{eq:4.3}
  \begin{split}
    &\Big(\mathcal{C}_x\mathcal{C}_y-\frac{K_1^{+}\tau}{2}\mathcal{C}_y{_L}\mathcal{D}_{h,p,q}^{\alpha}
        -\frac{K_2^{+}\tau}{2}\mathcal{C}_y{_R}\mathcal{D}_{h,p,q}^{\alpha}
        -\frac{K_1^{-}\tau}{2}\mathcal{C}_x{_L}\mathcal{D}_{h,p,q}^{\beta}
        -\frac{K_2^{-}\tau}{2}\mathcal{C}_x{_R}\mathcal{D}_{h,p,q}^{\beta}\Big)u_{i,j}^{n+1} \\
    =&\Big(\mathcal{C}_x\mathcal{C}_y+\frac{K_1^{+}\tau}{2}\mathcal{C}_y{_L}\mathcal{D}_{h,p,q}^{\alpha}
        +\frac{K_2^{+}\tau}{2}\mathcal{C}_y{_R}\mathcal{D}_{h,p,q}^{\alpha}
        +\frac{K_1^{-}\tau}{2}\mathcal{C}_x{_L}\mathcal{D}_{h,p,q}^{\beta}
        +\frac{K_2^{-}\tau}{2}\mathcal{C}_x{_R}\mathcal{D}_{h,p,q}^{\beta}\Big)u_{i,j}^{n}\\
     & +\tau \mathcal{C}_x\mathcal{C}_yf_{i,j}^{n+1/2}+\tau\varepsilon_{i,j}^{n+1/2},
  \end{split}
\end{equation}
where
\begin{equation}
  \begin{split}
    \varepsilon_{i,j}^{n+1/2}=&-h^3\Big(K_1^{+}c_{p,q,3}^{\alpha}\,{_a}D_x^{\alpha+3}u
     +K_2^{+}c_{p,q,3}^{\alpha}\,{_x}D_b^{\alpha+3}u
     +K_1^{-}c_{p,q,3}^{\beta}\,{_c}D_y^{\beta+3}u  \\
    &+K_2^{-}c_{p,q,3}^{\beta}\,{_y}D_d^{\beta+3}u\Big)_{i,j}^{n+1/2}+O(\tau^2+h^4).
  \end{split}
\end{equation}
And we denote
\begin{equation*}
  \delta_x^\alpha=K_1^{+}{_L}\mathcal{D}_{h,p,q}^{\alpha}+K_2^{+}{_R}\mathcal{D}_{h,p,q}^{\alpha}, \qquad
  \delta_y^\beta=K_1^{-}{_L}\mathcal{D}_{h,p,q}^{\beta}+K_2^{-}{_R}\mathcal{D}_{h,p,q}^{\beta}.
\end{equation*}
Using the Taylor's expansion shows that
\begin{equation}\label{eq:4.4}
 \frac{\tau^2}{4}\delta_x^{\alpha}\delta_y^{\beta}(u_{i,j}^{n+1}-u_{i,j}^{n})=\frac{\tau^3}{4}\Big((K_1^{+}{_a}D_x^{\alpha}+K_2^{+}{_x}D_{b}^{\alpha})(K_1^{-}{_c}D_{y}^{\beta}+K_2^{-}{_y}D_{d}^{\beta})\frac{\partial u}{\partial t}\Big)_{i,j}^{n+1/2}+O(\tau^5+\tau^3h^2).
 \end{equation}
Adding \ref{eq:4.4} to the corresponding sides of \ref{eq:4.3}, then
we can factorize it as
\begin{equation}
  \begin{split}\label{eq:4.6}
    \Big(\mathcal{C}_x-\frac{\tau}{2}\delta_x^\alpha\Big)\Big(\mathcal{C}_y-\frac{\tau}{2}\delta_y^\beta\Big)u_{i,j}^{n+1}
    = \Big(\mathcal{C}_x+\frac{\tau}{2}\delta_x^\alpha\Big)\Big(\mathcal{C}_y+\frac{\tau}{2}\delta_y^\beta\Big)u_{i,j}^n+\tau \mathcal{C}_x\mathcal{C}_yf_{i,j}^{n+1/2}+\tau \hat{\varepsilon}_{i,j}^{n+1/2},
  \end{split}
\end{equation}
where
\begin{equation}
\hat{\varepsilon}_{i,j}^{n+1/2}=\varepsilon_{i,j}^{n+1/2}+\frac{\tau^2}{4}\Big((K_1^{+}{_a}D_x^{\alpha}+K_2^{+}{_x}D_{b}^{\alpha})(K_1^{-}{_c}D_{y}^{\beta}+K_2^{-}{_y}D_{d}^{\beta})\frac{\partial u}{\partial t}\Big)_{i,j}^{n+1/2}+O(\tau^2+\tau^2h^2+h^4).
\end{equation}
By denoting $U_{i,j}^n$ as the numerical approximation to
$u_{i,j}^n$, the compact finite difference scheme for \ref{eq:4.1}
is
\begin{equation}
  \begin{split}\label{eq:4.7}
    & \Big(\mathcal{C}_x-\frac{\tau}{2}\delta_x^\alpha\Big)\Big(\mathcal{C}_y-\frac{\tau}{2}\delta_y^\beta\Big)U_{i,j}^{n+1}
    = \Big(\mathcal{C}_x+\frac{\tau}{2}\delta_x^\alpha\Big)\Big(\mathcal{C}_y+\frac{\tau}{2}\delta_y^\beta\Big)U_{i,j}^n+\tau \mathcal{C}_x\mathcal{C}_yf_{i,j}^{n+1/2}.
  \end{split}
\end{equation}
Introducing the intermediate variable $V_{i,j}^n$, we derive several splitting schemes, such as the compact Peaceman-Rachford ADI scheme:
\begin{subequations}
\label{eq:4.16}
\begin{align}
    &\Big(\mathcal{C}_x-\frac{\tau}{2}\delta_x^\alpha\Big)V_{i,j}^{n}~~~
     =\Big(\mathcal{C}_y+\frac{\tau}{2}\delta_y^\beta\Big)U_{i,j}^{n}
     +\frac{\tau}{2}\mathcal{C}_yf_{i,j}^{n+1/2}, \\
    &\Big(\mathcal{C}_y-\frac{\tau}{2}\delta_y^\beta\Big)U_{i,j}^{n+1}
     =\Big(\mathcal{C}_x+\frac{\tau}{2}\delta_x^\alpha\Big)V_{i,j}^{n}
     +\frac{\tau}{2}\mathcal{C}_yf_{i,j}^{n+1/2},
\end{align}
\end{subequations}
the compact Douglas ADI scheme:
\begin{subequations}\label{eq:4.17}
\begin{align}
    &\Big(\mathcal{C}_x-\frac{\tau}{2}\delta_x^\alpha\Big)V_{i,j}^{n}~~~=
     \Big(\mathcal{C}_x\mathcal{C}_y+\frac{\tau}{2}\mathcal{C}_y\delta_x^\alpha
     +\tau\mathcal{C}_x\delta_y^\beta\Big)U_{i,j}^{n}+\tau \mathcal{C}_x\mathcal{C}_yf_{i,j}^{n+1/2}, \\
    &\Big(\mathcal{C}_y-\frac{\tau}{2}\delta_y^\beta\Big)U_{i,j}^{n+1}
     =V_{i,j}^n-\frac{\tau}{2}\delta_y^\beta U_{i,j}^n,
\end{align}
\end{subequations}
and the compact D'Yakonov ADI scheme:
\begin{subequations}\label{eq:4.18}
\begin{align}
    &\Big(\mathcal{C}_x-\frac{\tau}{2}\delta_x^\alpha\Big)V_{i,j}^{n}~~~
     =\Big(\mathcal{C}_x+\frac{\tau}{2}\delta_x^\alpha\Big)
     \Big(\mathcal{C}_y+\frac{\tau}{2}\delta_y^\beta\Big)U_{i,j}^{n}+\tau \mathcal{C}_x\mathcal{C}_yf_{i,j}^{n+1/2}, \\
    &\Big(\mathcal{C}_y-\frac{\tau}{2}\delta_y^\beta\Big)U_{i,j}^{n+1}=V_{i,j}^n.
\end{align}
\end{subequations}
A simple calculation shows that
\begin{equation}
\label{eq:4.8}
\frac{\tau^3}{4}\delta_x^{\alpha}\delta_y^{\beta}f_{i,j}^{n+1/2}=\frac{\tau^3}{4}(K_1^{+}{_a}D_x^{\alpha}+K_2^{+}{_x}D_{b}^{\alpha})(K_1^{-}{_c}D_{y}^{\beta}+K_2^{-}{_y}D_{d}^{\beta})f_{i,j}^{n+1/2}+O(\tau^3h^2).
\end{equation}
Then from \ref{eq:4.6} and \ref{eq:4.8}, it yields that
\begin{equation}
  \begin{split}\label{eq:4.9}
    \Big(\mathcal{C}_x-\frac{\tau}{2}\delta_x^\alpha\Big)
    \Big(\mathcal{C}_y-\frac{\tau}{2}\delta_y^\beta\Big)u_{i,j}^{n+1}
    = & \Big(\mathcal{C}_x+\frac{\tau}{2}\delta_x^\alpha\Big)
    \Big(\mathcal{C}_y+\frac{\tau}{2}\delta_y^\beta\Big)u_{i,j}^n \\
    & +\tau \mathcal{C}_x\mathcal{C}_yf_{i,j}^{n+1/2}+
    \frac{\tau^3}{4}\delta_x^{\alpha}\delta_y^{\beta}f_{i,j}^{n+1/2}+\tau \tilde{\varepsilon}_{i,j}^{n+1/2}.
  \end{split}
\end{equation}
where
\begin{equation}
\tilde{\varepsilon}_{i,j}^{n+1/2}=\hat{\varepsilon}_{i,j}^{n+1/2}-\frac{\tau^2}{4}(K_1^{+}{_a}D_x^{\alpha}+K_2^{+}{_x}D_{b}^{\alpha})(K_1^{-}{_c}D_{y}^{\beta}+K_2^{-}{_y}D_{d}^{\beta})f_{i,j}^{n+1/2}+O(\tau^2+\tau^2h^2).
\end{equation}
Eliminating the truncating error and denoting $U_{i,j}^n$ as the numerical approximation of $u_{i,j}^n$, we have
\begin{equation}
  \begin{split}\label{eq:4.10}
    \Big(\mathcal{C}_x-\frac{\tau}{2}\delta_x^\alpha\Big)
    \Big(\mathcal{C}_y-\frac{\tau}{2}\delta_y^\beta\Big)U_{i,j}^{n+1}
    = \Big(\mathcal{C}_x+\frac{\tau}{2}\delta_x^\alpha\Big)
    \Big(\mathcal{C}_y+\frac{\tau}{2}\delta_y^\beta\Big)U_{i,j}^n+\tau \mathcal{C}_x\mathcal{C}_yf_{i,j}^{n+1/2}+
    \frac{\tau^3}{4}\delta_x^{\alpha}\delta_y^{\beta}f_{i,j}^{n+1/2}.
  \end{split}
\end{equation}
Introducing the intermediate variable $V_{i,j}^{n}$, we obtain the compact locally one-dimensional (LOD) scheme,
\begin{subequations}\label{eq:4.15}
\begin{align}
    &\Big(\mathcal{C}_x-\frac{\tau}{2}\delta_x^\alpha\Big)V_{i,j}^{n}~~~
      =\Big(\mathcal{C}_x+\frac{\tau}{2}\delta_x^\alpha\Big)U_{i,j}^{n}
      +\frac{\tau}{2}\Big(\mathcal{C}_x+\frac{\tau}{2}\delta_x^\alpha\Big)f_{i,j}^{n+1/2},\\
    &\Big(\mathcal{C}_y-\frac{\tau}{2}\delta_y^\beta\Big)U_{i,j}^{n+1}
      =\Big(\mathcal{C}_y+\frac{\tau}{2}\delta_y^\beta\Big)V_{i,j}^{n}
      +\frac{\tau}{2}\Big(\mathcal{C}_y-\frac{\tau}{2}\delta_y^\beta\Big)f_{i,j}^{n+1/2}.
\end{align}
\end{subequations}
\subsection{Stability and Convergence}
Let
 \begin{equation*}
  V_h=\{v:v=\{v_{i,j}\} \text{ is a grid function in } \Omega_h \text{~and~}v_{i,j}=0 \text{ on } \Gamma_h \},
\end{equation*}
where
\begin{align*}
&\Omega_h=\{(i, j): 1\le i\le N_x-1, 1\le j\le N_y-1\},   \\
&\Gamma_h=\{(i, j): i=0,N_x; 0\le j\le N_y\}\cup\{(i, j): 0\le i\le N_x; j=0, N_y\}.
\end{align*}
For any $v=\{v_{i}\}\in V_h$, we define its pointwise maximum norm
 and discrete $L^2$ norm as
\begin{equation}
  \|v\|_{\infty}=\max\limits_{(i,j)\in\Omega_h} |v_{i,j}|, \quad \|v\|=\sqrt{h^2\sum_{i=1}^{N_x-1}\sum_{j=1}^{N_y-1}v_{i,j}^2}.
\end{equation}
In the following, we first list some properties of Kronecker
products of matrices.
\begin{lemma}[\cite{Laub:05}]\label{lem:4}
  Let $A\in\mathbb{R}^{n\times n}$ have eigenvalues $\{\lambda_i\}_{i=1}^n$, and $B\in\mathbb{R}^{m\times m}$ have eigenvalues $\{\mu_j\}_{j=1}^m$. Then the $mn$ eigenvalues of $A\otimes B$, which represents the kronecker product of matrix $A$ and $B$, are
 \begin{equation*}
    \lambda_1\mu_1,\ldots,\lambda_1\mu_m,\lambda_2\mu_1,\ldots,\lambda_2\mu_m,\ldots,\lambda_n\mu_1,\ldots,\lambda_n\mu_m.
  \end{equation*}
\end{lemma}
\begin{lemma}[\cite{Laub:05}]\label{lem:5}
  Let $A\in \mathbb{R}^{m\times n}, B\in \mathbb{R}^{r\times s}, C\in \mathbb{R}^{n\times p}, D\in \mathbb{R}^{s\times t}$. Then
  \begin{equation}
    (A\otimes B)(C\otimes D)=AC\otimes BD~~(\in \mathbb{R}^{mr\times pt}).
  \end{equation}
  Moreover, if $A, B\in \mathbb{R}^{n\times n}$, $I$ is a unit matrix of order $n$, then matrices $I\otimes A$ and $B\otimes I$ commute.
\end{lemma}
\begin{lemma}[\cite{Laub:05}]\label{lem:6}
  For all $A$ and $B$, $(A\otimes B)^{\mathrm{T}}=A^{\mathrm{T}}\otimes B^{\mathrm{T}}$ and
  $(A\otimes B)^{-1}=A^{-1}\otimes B^{-1}$ if $A$ and $B$ are invertible.
\end{lemma}
\begin{lemma}[\cite{Zhang:11}]\label{lem:7}
  Let $A, B$ be two symmetric and positive semi-definite matrices, symbolized $A\geq 0$ and $B\geq 0$. Then $A\otimes B\geq 0$.
\end{lemma}
\begin{theorem}\label{thm:6}
  For the case of $(p, q)=(1, 0)$, the difference schemes \ref{eq:4.7} and \ref{eq:4.10} are unconditionally stable for $1<\alpha,\beta\le2$. And for the case of $(p, q)=(1, -1)$, the difference schemes \ref{eq:4.7} and \ref{eq:4.10} are also unconditionally stable when $\frac{1+\sqrt{73}}{6}\le\alpha,\beta\le2$.
\end{theorem}
\begin{proof}
  We express grid function $U_{i,j}^{n}$ in the vector form as
  \begin{align*}
    & U^{n}=(u_{1,1}^{n},u_{2,1}^{n},\cdots,u_{N_x-1,1}^{n},u_{1,2}^{n},u_{2,2}^{n},
    \cdots,u_{N_x-1,2}^{n},\cdots,u_{1,N_y-1}^{n},u_{2,N_y-1}^n,
    \cdots,u_{N_x-1,N_y-1}^{n})^{\mathrm{T}}, 
  \end{align*}
  and denote
  \begin{align}
    &C_x=I_y\otimes C_{\alpha}, ~~~~ C_y=C_{\beta}\otimes I_x,   \label{eq:4.13a}\\
    &\mathcal{D}_x=\frac{K_1^{+}\tau}{2h^{\alpha}}I_y\otimes A_{\alpha}+\frac{K_2^{+}\tau}{2h^{\alpha}}I_y\otimes A_{\alpha}^{\mathrm{T}},  ~~~~
    \mathcal{D}_y=\frac{K_1^{-}\tau}{2h^{\beta}}A_{\beta}\otimes I_x+\frac{K_2^{-}\tau}{2h^{\beta}}A_{\beta}^{\mathrm{T}}\otimes I_x, \label{eq:4.13b}
  \end{align}
  where the symbol $\otimes$ denotes the Kronecker product, $I_x$ and $I_y$ are unit matrices of $(N_x-1)$ and $(N_y-1)$ squares, respectively, and the matrices $A_{\alpha}$ and $A_{\beta}$ are defined in \ref{eq:2.9} corresponding to $\alpha$ and $\beta$, $C_{\alpha}$ and $C_{\beta}$ are given in \ref{eq:2.20} with coefficients $\alpha$ and $\beta$. Therefore, denoting the disturbances of $U^{n+1}$ and $U^{n}$ by $\delta U^{n+1}$ and $\delta U^n$, respectively, we have from \ref{eq:4.7} and \ref{eq:4.10}  that
  \begin{equation}\label{eq:4.11}
    \delta U^{n+1}=\big(C_y-\mathcal{D}_y\big)^{-1}\big(C_x-\mathcal{D}_x\big)^{-1}
    \big(C_x+\mathcal{D}_x\big)\big(C_y+\mathcal{D}_y\big)\delta U^{n}.
  \end{equation}
  Using Lemma \ref{lem:5}, we can check that $C_x$ and $\mathcal{D}_x$ commute with $C_y$ and $\mathcal{D}_y$, which deduces that $(C_y-\mathcal{D}_y)^{-1}$ and $(C_y+\mathcal{D}_y)$ commute with $(C_x-\mathcal{D}_x)^{-1}$ and $(C_x+\mathcal{D}_x)$. Then it obtains from \ref{eq:4.11} that
  \begin{equation}
    \delta U^{n}=\Big(\big(C_y-\mathcal{D}_y\big)^{-1}\big(C_y+\mathcal{D}_y\big)\Big)^n
    \Big(\big(C_x-\mathcal{D}_x\big)^{-1}\big( C_x+\mathcal{D}_x\big)\Big)^n\delta U^{0}.
  \end{equation}
 From Remark \ref{rem:3}, Lemma \ref{lem:4} and Lemma \ref{lem:5}, we have that $C_x$ and $C_y$ are symmetric and
 positive definite matrices in the cases of $(p, q)=(1, 0)$ with $1<\alpha, \beta \le 2$ and $(p, q)=(1, -1)$
 with $\frac{1+\sqrt{73}}{6}\le\alpha, \beta\le 2$, which yields that $C_x^{-1}$ and $C_y^{-1}$ are also symmetric and positive definite.
 On the other hand, Lemma \ref{thm:4} and \ref{lem:6} indicate that the eigenvalues of $\frac{A_{\alpha}+A_{\alpha}^{\mathrm{T}}}{2}$ and $\frac{A_{\beta}+A_{\beta}^{\mathrm{T}}}{2}$ are all negative when $1<\alpha,\beta\le2$, then employing Lemma \ref{lem:4}, we obtain that $(\mathcal{D}_x+\mathcal{D}_x^{\mathrm{T}})$ and $(\mathcal{D}_y+\mathcal{D}_y^{\mathrm{T}})$ are both symmetric and negative definite matrices. Then it yields that $v^{\mathrm{T}}(\mathcal{D}_x+\mathcal{D}_x^{\mathrm{T}})v<0$ and $v^{\mathrm{T}}(\mathcal{D}_y+\mathcal{D}_y^{\mathrm{T}})v<0$ hold for any non-zero vector $v\in \mathbb{R}^{(N_x-1)(N_y-1)}$, and
  \begin{equation}
    v^{\mathrm{T}}\Big((C_\gamma^{-1}\mathcal{D}_\gamma)C_\gamma^{-1}
    +C_\gamma^{-1}(C_\gamma^{-1}\mathcal{D}_\gamma)^{\mathrm{T}}  \Big)v
    =v^{\mathrm{T}}C_\gamma^{-1}(\mathcal{D}_\gamma
    +\mathcal{D}_\gamma^{\mathrm{T}})C_\gamma^{-1}v < 0,  ~~~\gamma=x,y,
  \end{equation}
  which means that the matrix $(C_\gamma^{-1}\mathcal{D}_\gamma)C_\gamma^{-1}+C_\gamma^{-1} (C_\gamma^{-1}\mathcal{D}_\gamma)^{\mathrm{T}}$
   for $\gamma=x, y$ are symmetric and negative definite matrices,
   then it implies from Lemma \ref{thm:2} that the real parts of all the eigenvalues $\{\lambda_\gamma\}$ of $C_\gamma^{-1}\mathcal{D}_\gamma$ for $\gamma=x, y$ are negative, and $|\frac{1+\lambda_\gamma}{1-\lambda_\gamma}|<1$. Additionally, $\lambda_\gamma$ is an eigenvalue of $C_\gamma^{-1}\mathcal{D}_\gamma$ if and only if $\frac{1-\lambda_\gamma}{1+\lambda_\gamma}$ is an eigenvalue of $(I-C_\gamma^{-1}\mathcal{D}_\gamma)^{-1}(I+C_\gamma^{-1}\mathcal{D}_\gamma)$, thus the spectral radius of each matrix is less than 1, which concludes that $\big((I-C_x^{-1}\mathcal{D}_x)^{-1}(I+C_x^{-1}\mathcal{D}_x)\big)^n$ and $\big((I-C_y^{-1}\mathcal{D}_y)^{-1}(I+C_y^{-1}\mathcal{D}_y)\big)^n$ converge to zero matrix, therefore, the difference scheme \ref{eq:4.7} and \ref{eq:4.10} are stable.
\end{proof}
\begin{lemma}[\cite{Zhang:11}]\label{lem:9}
  Let $A$ be an $n$-square symmetric and positive semi-definite matrix. Then there is a unique $n$-square symmetric and positive semi-definite matrix $B$ such that $B^2=A$. Such a matrix $B$ is called the square root of $A$, denoted by $A^{1/2}$.
\end{lemma}
\begin{theorem}\label{thm:7}
Let $u_{i,j}^{n}$ be the exact solution of \ref{eq:4.1}, and
$U_{i,j}^{n}$ be the solution of the difference schemes \ref{eq:4.7}
or \ref{eq:4.10}, then in the cases of $(p, q)=(1, 0)$ with
$1<\alpha, \beta<2$ and $(p, q)=(1, -1)$ with
$\frac{1+\sqrt{73}}{6}<\alpha, \beta<2$, we have
  \begin{equation}\label{eq:4.27}
    \|u^n-U^n\|\leq c(\tau^2+h^3), ~~1\le n\le M,
  \end{equation}
  where c denotes a positive constant and $\|\cdot\|$ stands for the discrete $L^2$-norm.
\end{theorem}
\begin{proof}
  Let $e_{i,j}^{n}=u_{i,j}^{n}-U_{i,j}^{n}$, subtracting \ref{eq:4.6} from
  \ref{eq:4.7} leads to
  \begin{equation}\label{eq:4.12}
      \big(C_x-\mathcal{D}_x\big)\big(C_y-\mathcal{D}_y\big)e^{n+1}
      =\big(C_x+\mathcal{D}_x\big)\big(C_y+\mathcal{D}_y\big)e^{n}
      +\tau \mathcal{E}^{n+1/2},
  \end{equation}
  where
  \begin{align*}
    &e=(e_{1,1},e_{2,1},\cdots,e_{N_x-1,1},e_{1,2},e_{2,2},\cdots,e_{N_x-1,2},\cdots,
        e_{1,N_y-1},e_{2,N_y-1},\cdots,e_{N_x-1,N_y-1})^{\mathrm{T}},    \\
    &\mathcal{E}=(\hat{\varepsilon}_{1,1},\hat{\varepsilon}_{2,1},\cdots,
        \hat{\varepsilon}_{N_x-1,1},\hat{\varepsilon}_{1,2},\hat{\varepsilon}_{2,2},\cdots,
        \hat{\varepsilon}_{N_x-1,2},\cdots,\hat{\varepsilon}_{1,N_y-1},
        \hat{\varepsilon}_{2,N_y-1},\cdots,\hat{\varepsilon}_{N_x-1,N_y-1})^{\mathrm{T}},
  \end{align*}
and the matrices $C_x, C_y$ and $\mathcal{D}_x, \mathcal{D}_y$ are
given by \ref{eq:4.13a} and \ref{eq:4.13b}, respectively.

As stated in Theorem \ref{thm:6}, under the cases of $(p, q)=(1, 0)$
with $1< \alpha, \beta\le 2$ and $(p, q)=(1, -1)$ with
$\frac{1+\sqrt{73}}{6}<\alpha, \beta\le 2$, the matrices $C_\alpha$
and $C_\beta$ and their inverse are symmetric and positive definite.
And from Lemma \ref{lem:6} and \ref{lem:9}, we know that
$(C_x^{-1})^{1/2}=I_y\otimes (C_\alpha^{-1})^{1/2}$ and
$(C_y^{-1})^{1/2}=(C_\beta^{-1})^{1/2}\otimes I_x$ uniquely exist
and are symmetric and positive semi-definite matrices. Then
multiplying \ref{eq:4.12} by $(C_x^{-1})^{1/2}(C_y^{-1})^{1/2}$, and
making the discrete $L^2$-norm on both sides, we have
  \begin{equation}\label{eq:4.29}
    \begin{split}
      &\|(C_x^{-1})^{1/2}(C_y^{-1})^{1/2}(C_x-\mathcal{D}_x)(C_y-\mathcal{D}_y)e^{n+1}\| \\
      &~~~~\le\|(C_x^{-1})^{1/2}(C_y^{-1})^{1/2}(C_x+\mathcal{D}_x)(C_y+\mathcal{D}_y)e^{n}\|
      +\tau\|(C_x^{-1})^{1/2}(C_y^{-1})^{1/2}\mathcal{E}^{n+1/2}\|.
    \end{split}
  \end{equation}
Simple calculations show that $(C_y-\mathcal{D}_y)$ commutes with
$(C_x-\mathcal{D}_x), (C_x^{-1})^{1/2},
(C_x-\mathcal{D}_x^\mathrm{T})$; $(C_y+\mathcal{D}_y)$ commutes with
$(C_x+\mathcal{D}_x), (C_x^{-1})^{1/2},
(C_x+\mathcal{D}_x^\mathrm{T})$; and $(C_y^{-1})^{1/2}$ commutes
with $(C_x^{-1})^{1/2}, (C_x\pm\mathcal{D}_x^\mathrm{T})$. From
Lemma \ref{thm:4} and \ref{lem:7}, we have that
$\mathcal{D}_\gamma+\mathcal{D}_\gamma^{\mathrm{T}} ~(\gamma=x, y)$
are symmetric and negative definite matrices. Thus, using Lemma
\ref{lem:5} and Lemma \ref{lem:7}, we obtain that
  \begin{equation}\label{eq:4.14a}
    \begin{split}
      &\Big((C_x^{-1})^{1/2}(C_y^{-1})^{1/2}(C_x-\mathcal{D}_x)(C_y-\mathcal{D}_y)\Big)^{\mathrm{T}}
      \Big((C_x^{-1})^{1/2}(C_y^{-1})^{1/2}(C_x-\mathcal{D}_x)(C_y-\mathcal{D}_y)\Big)\\
      &~~~~~~=\big(C_y-\mathcal{D}_y^{\mathrm{T}}-\mathcal{D}_y+\mathcal{D}_y^{\mathrm{T}}
      C_y^{-1}\mathcal{D}_y\big)\big(C_x-\mathcal{D}_x^{\mathrm{T}}-\mathcal{D}_x
      +\mathcal{D}_x^{\mathrm{T}}C_x^{-1}\mathcal{D}_x\big)\\
      &~~~~~~\ge\big(C_y+\mathcal{D}_y^{\mathrm{T}}C_y^{-1}\mathcal{D}_y\big)
      \big(C_x+\mathcal{D}_x^{\mathrm{T}}C_x^{-1}\mathcal{D}_x\big)
      +\big(\mathcal{D}_y^{\mathrm{T}}+\mathcal{D}_y\big)
      \big(\mathcal{D}_x^{\mathrm{T}}+\mathcal{D}_x\big),
    \end{split}
  \end{equation}
  and
  \begin{equation}\label{eq:4.14b}
    \begin{split}
      &\Big((C_x^{-1})^{1/2}(C_y^{-1})^{1/2}(C_x+\mathcal{D}_x)(C_y+\mathcal{D}_y)\Big)^{\mathrm{T}}
      \Big((C_x^{-1})^{1/2}(C_y^{-1})^{1/2}(C_x+\mathcal{D}_x)(C_y+\mathcal{D}_y)\Big)\\
      &~~~~~~=\big(C_y+\mathcal{D}_y^{\mathrm{T}}+\mathcal{D}_y+\mathcal{D}_y^{\mathrm{T}}
      C_y^{-1}\mathcal{D}_y\big)\big(C_x+\mathcal{D}_x^{\mathrm{T}}+\mathcal{D}_x
      +\mathcal{D}_x^{\mathrm{T}}C_x^{-1}\mathcal{D}_x\big)\\
      &~~~~~~\le\big(C_y+\mathcal{D}_y^{\mathrm{T}}C_y^{-1}\mathcal{D}_y\big)
      \big(C_x+\mathcal{D}_x^{\mathrm{T}}C_x^{-1}\mathcal{D}_x\big)
      +\big(\mathcal{D}_y^{\mathrm{T}}+\mathcal{D}_y\big)
      \big(\mathcal{D}_x^{\mathrm{T}}+\mathcal{D}_x\big),
    \end{split}
  \end{equation}
  where the matrices $A\ge B$ means that $(A-B)$ is positive semi-definite.
  And define
  \begin{equation}\label{eq:4.19}
    E^{n}=\sqrt{h^2(e^n)^{\mathrm{T}}\Big(\big(C_y+\mathcal{D}_y^{\mathrm{T}}C_y^{-1}\mathcal{D}_y\big)
      \big(C_x+\mathcal{D}_x^{\mathrm{T}}C_x^{-1}\mathcal{D}_x\big)
    +\big(\mathcal{D}_y^{\mathrm{T}}+\mathcal{D}_y\big)
      \big(\mathcal{D}_x^{\mathrm{T}}+\mathcal{D}_x\big)\Big)(e^n)},
  \end{equation}
it concludes from \ref{eq:4.13a}, \ref{eq:4.13b}, Lemma \ref{lem:5},
 Lemma \ref{lem:7}, and Lemma \ref{thm:4} that the matrices
$C_y\mathcal{D}_x^{\mathrm{T}}C_x^{-1}\mathcal{D}_x$,
$C_x\mathcal{D}_y^{\mathrm{T}}C_y^{-1}\mathcal{D}_y$,
$\mathcal{D}_y^{\mathrm{T}}C_y^{-1}\mathcal{D}_y
  \mathcal{D}_x^{\mathrm{T}}C_x^{-1}\mathcal{D}_x$
  and $\big(\mathcal{D}_y^{\mathrm{T}}+\mathcal{D}_y\big)
  \big(\mathcal{D}_x^{\mathrm{T}}+\mathcal{D}_x\big)$ are all symmetric and positive definite, which follows that
  \begin{equation}\label{eq:4.22}
    E^n\ge\sqrt{h^2(e^n)^\mathrm{T}C_xC_y(e^n)}
    =\sqrt{h^2(e^n)^\mathrm{T}\big(C_\beta\otimes C_\alpha\big)(e^n)}\ge\sqrt{\lambda_{\min}(C_\alpha)\lambda_{\min}(C_\beta)}\|e^n\|,
  \end{equation}
  where $\lambda_{\min}(C_\alpha)$ and $\lambda_{\min}(C_\beta)$ are the minimum eigenvalue of matrix $C_\alpha$ and $C_\beta$, respectively. As stated in Remark \ref{rem:3}, $\lambda_{\min}(C_\alpha)>1-4c_{1,-1,2}^{\alpha}>0,~\lambda_{\min}(C_\beta)>1-4c_{1,-1,2}^{\beta}>0$ if $\frac{1+\sqrt{73}}{6}<\alpha,\beta\le 2$ and $(p,q)=(1,-1)$; $\lambda_{\min}(C_\alpha),\lambda_{\min}(C_\beta)>\frac{23}{72}$ if $1\le \alpha,\beta\le 2$ and $(p,q)=(1,0)$.
  Together with \ref{eq:4.14a} and \ref{eq:4.14b}, then \ref{eq:4.29} becomes as
  \begin{equation}\label{eq:4.20}
    E^{k+1}-E^{k}\le\tau\|(C_x^{-1})^{1/2}(C_y^{-1})^{1/2}\mathcal{E}^{k+1/2}\|.
  \end{equation}
  From the Rayleigh-Ritz Theorem (see Theorem 8.8 in \cite{Zhang:11}) and Lemma \ref{lem:4}, we have for $k=0,\cdots, n-1$ that
  \begin{equation}\label{eq:4.33}
    \begin{split}
      \|(C_x^{-1})^{1/2}(C_y^{-1})^{1/2}\mathcal{E}^{k+1/2}\|
      &=\sqrt{h^2(\mathcal{E}^{k+1/2})^{\mathrm{T}}\big(C_x^{-1}C_y^{-1}\big)
        \mathcal{E}^{k+1/2}}\\
      &\le\sqrt{\lambda_{\max}(C_x^{-1}C_y^{-1})}\|\mathcal{E}^{k+1/2}\|
      =\frac{1}{\sqrt{\lambda_{\min}(C_xC_y)}}\|\mathcal{E}^{k+1/2}\|\\
      &=\frac{1}{\sqrt{\lambda_{\min}(C_\alpha)\lambda_{\min}(C_\beta)}}
      \|\mathcal{E}^{k+1/2}\|.
    \end{split}
  \end{equation}
  Summing up \ref{eq:4.20} for all $0\le k\le n-1$ shows that
  \begin{equation}\label{eq:4.21}
    E^n\le \tau \sum_{k=0}^{n-1}\|(C_x^{-1})^{1/2}(C_y^{-1})^{1/2}\mathcal{E}^{k+1/2}\|
    \le \frac{\tau}{\sqrt{\lambda_{\min}(C_\alpha)\lambda_{\min}(C_\beta)}}
      \sum_{k=0}^{n-1}\|\mathcal{E}^{k+1/2}\|.
  \end{equation}
Combining \ref{eq:4.21} and \ref{eq:4.22}, and noticing
$|\hat{\varepsilon}_{i,j}^{k+1/2}|\le \tilde{c}(\tau^2+h^3)$ for all
$1\le i, j \le N-1$, we obtain
  \begin{equation}
    \|e^n\| \leq \frac{c~T}{\lambda_{\min}(C_\alpha)\lambda_{\min}(C_\beta)}(\tau^2+h^3).
  \end{equation}
The estimate for scheme \ref{eq:4.10} can also be obtained by the
similar approach used above.
\end{proof}

\begin{remark}
If $\alpha,\beta=1,2$ with $(p,q)=(1,0)$ and $\alpha,\beta=2$ with
$(p,q)=(1,-1)$, then by the reasoning of the proof of Theorem
\ref{thm:7}, we obtain the following error estimate for the
difference scheme \ref{eq:4.7} and \ref{eq:4.10}
  \begin{equation}
    \|u^n-U^n\|\leq c(\tau^2+h^4), ~~~1\leq n\leq M.
  \end{equation}
\end{remark}
\section{Numerical Experiments}
In this section, the numerical results of one and two dimensional
cases are presented to show the effectiveness and convergence orders
of the schemes.

For saving computational time, we do one extrapolation to increase
the accuracy to the third order in time (see \cite{Marchuk:1983}).
The detailed extrapolation algorithm is described as follows.

  Step 1. Calculate $\zeta_1, \zeta_2$ from the following linear algebraic equations,
    \begin{equation*}
      \begin{cases}
        \zeta_1+\zeta_2=1,&  \\
        \zeta_1+{\displaystyle\frac{\zeta_2}{4}}=0;&
      \end{cases}
    \end{equation*}

  Step 2. Compute the solution $U^n$ of the compact difference schemes with two time stepsizes $\tau$ and $\tau/2$;

  Step 3. Evaluate the extrapolation solution $W^{n}(\tau)$ by
    \begin{equation*}
      W^n(\tau)=\zeta_1~U^{n}(\tau)+\zeta_2~U^{n}(\tau/2).
    \end{equation*}
\begin{example}
\label{Exam:1}
  Consider the following one dimensional space fractional diffusion equation
  \begin{equation}
  \label{eq:5.2}
    \begin{split}
      & \frac{\partial u(x,t)}{\partial t}={_0}D_x^{\alpha}u(x,t)+{_x}D_1^{\alpha}u(x,t)
        +f(x,t), \quad (x,t)\in (0,1)\times(0,1],\\
      & u(0,t)=u(1,t)=0, \quad t\in [0, 1],\\
      & u(x,0)=x^3(1-x)^3, \quad x\in [0, 1],
    \end{split}
  \end{equation}
  with the source term
  \begin{equation*}
    \begin{split}
    f(x,t)=-\mathrm{e}^{-t}\Big(x^3(1-x)^3+\frac{\Gamma(4)}{\Gamma(4-\alpha)}\big(x^{3-\alpha}+(1-x)^{3-\alpha}\big)
                    -3\frac{\Gamma(5)}{\Gamma(5-\alpha)}\big(x^{4-\alpha}+(1-x)^{4-\alpha}\big)\\
                    +3\frac{\Gamma(6)}{\Gamma(6-\alpha)}\big(x^{5-\alpha}+(1-x)^{5-\alpha}\big)
                     -\frac{\Gamma(7)}{\Gamma(7-\alpha)}\big(x^{6-\alpha}+(1-x)^{6-\alpha}\big)\Big).
    \end{split}
  \end{equation*}
And the exact solution of \ref{eq:5.2} is given by $u(x,t)=\mathrm{e}^{-t}x^3(1-x)^3$.
\end{example}

\begin{table}[h!]\fontsize{10pt}{12pt}\selectfont
  \begin{center}
  \caption{The maximum and $L^2$ errors and corresponding convergence rates to \ref{eq:5.2} approximated by the compact difference scheme at $t=1$ for different $\alpha$ with $\tau=h$. }\vspace{5pt}
  \begin{tabular*}{\linewidth}{@{\extracolsep{\fill}}*{2}{r}*{8}{c}}
    \toprule
     & & \multicolumn{4}{c}{$(p,q)=(1,0)$} & \multicolumn{4}{c}{$(p,q)=(1,-1)$} \\
    \cline{3-6} \cline{7-10} \\[-8pt]
    $\alpha$ & $N$ & $\|u^n-W^n\|_{\infty}$ & rate & $\|u^n-W^n\|$ & rate
    & $\|u^n-W^n\|_{\infty}$ & rate & $\|u^n-W^n\|$ & rate\\
    \toprule
    1.2 & 8 & 7.24249E-05 & - & 4.12739E-05 & - & 1.67141E-01 & - & 1.17292E-01 & - \\
    & 16 & 9.86726E-06 & 2.88  & 6.00551E-06 & 2.78 & 6.57473E-01 & -1.98  & 4.14981E-01 & -1.82 \\
    & 32 & 1.41964E-06 & 2.80  & 8.38665E-07 & 2.84 & 1.20582E+04 & -14.16  & 7.00900E+03 & -14.04 \\
    & 64 & 1.91577E-07 & 2.89  & 1.13698E-07 & 2.88 & 6.67837E+11 & -25.72  & 3.06749E+11 & -25.38 \\
    & 128 & 2.52254E-08 & 2.92  & 1.50548E-08 & 2.92 & 6.55394E+25 & -46.48  & 2.42775E+25 & -46.17 \\
    & 256 & 3.26935E-09 & 2.95  & 1.95986E-09 & 2.94 & 7.87651E+50 & -83.31  & 2.34722E+50 & -83.00 \\
    \midrule
    $\frac{1+\sqrt{73}}{6}$ & 8 & 5.48070E-05 & - & 3.18317E-05 & - & 2.66168E-04 & - & 1.62093E-04 & - \\
    & 16 & 6.71566E-06 & 3.03  & 4.15788E-06 & 2.94 & 4.38053E-05 & 2.60  & 2.50890E-05 & 2.69 \\
    & 32 & 8.92036E-07 & 2.91  & 5.69588E-07 & 2.87 & 6.07692E-06 & 2.85  & 3.69218E-06 & 2.76 \\
    & 64 & 1.19966E-07 & 2.89  & 7.81948E-08 & 2.86 & 8.00832E-07 & 2.92  & 5.16694E-07 & 2.84 \\
    & 128 & 1.60292E-08 & 2.90  & 1.06194E-08 & 2.88 & 1.05299E-07 & 2.93  & 7.00251E-08 & 2.88 \\
    & 256 & 2.11675E-09 & 2.92  & 1.42376E-09 & 2.90 & 1.37085E-08 & 2.94  & 9.30500E-09 & 2.91 \\
    \midrule
    1.8 & 8 & 3.78749E-05 & - & 2.87085E-05 & - & 1.49798E-04 & - & 8.54691E-05 & - \\
    & 16 & 3.90569E-06 & 3.28  & 2.91599E-06 & 3.30 & 1.94926E-05 & 2.94  & 1.23638E-05 & 2.79 \\
    & 32 & 4.51151E-07 & 3.11  & 3.47233E-07 & 3.07 & 2.60919E-06 & 2.90  & 1.76317E-06 & 2.81 \\
    & 64 & 5.86615E-08 & 2.94  & 4.44726E-08 & 2.96 & 3.31182E-07 & 2.98  & 2.42232E-07 & 2.86 \\
    & 128 & 7.68726E-09 & 2.93  & 5.85045E-09 & 2.93 & 4.26392E-08 & 2.96  & 3.24700E-08 & 2.90 \\
    & 256 & 1.00982E-09 & 2.93  & 7.74722E-10 & 2.92 & 5.51021E-09 & 2.95  & 4.28918E-09 & 2.92 \\
    \toprule
  \end{tabular*}\label{tab:1}
  \end{center}
\end{table}
In Table \ref{tab:1}, we present the errors $\|u^n-W^n\|,
\|u^n-W^n\|_\infty$ and corresponding convergence orders with
different space stepsizes, where
$W_{i}^n(\tau)=-\frac{1}{3}U_{i}^n(\tau)+\frac{4}{3}U_{i}^n(\tau/2)$
is the extrapolation solution, and $U_{i}^n$ satisfies the compact
scheme \ref{eq:3.7}. It can be noted that for the case of $(p,
q)=(1,-1)$, the numerical results are neither stable nor convergent
when the order $\alpha$ is less than the critical value
$\frac{1+\sqrt{73}}{6}(\thickapprox1.59)$, which coincides with the
theoretical results. Figure \ref{fig:a} shows that the convergence
rates of the maximum and $L^2$ errors to \ref{eq:5.2} approximated
by the compact difference scheme at $t=1$ with $N=128$ for different
$\alpha$, where the convergence rates fall from $4$ to $3$ near
$\alpha=1$ and increase gradually with $\alpha$ from $1.9$ to $2$.
\begin{figure}[h!]
    \subfigure[$(p,q)=(1,0)$]{
      \begin{minipage}[t]{0.5\linewidth}
        \centering
        \includegraphics[scale=0.5]{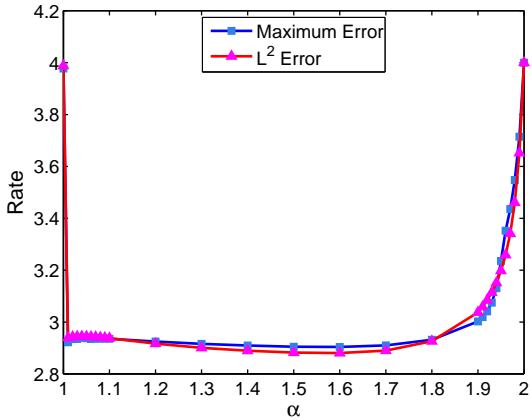}
      \end{minipage}}
    \subfigure[$(p,q)=(1,-1)$]{
      \begin{minipage}[t]{0.5\linewidth}
        \centering
        \includegraphics[scale=0.5]{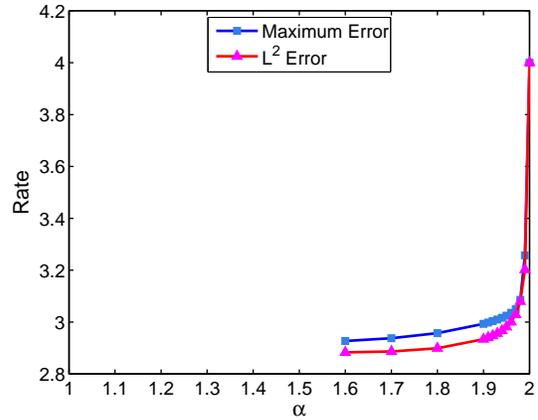}
      \end{minipage}}
    \caption{The convergence rates of the maximum and $L^2$ errors to \ref{eq:5.2} approximated by the compact difference scheme at $t=1$ with $N=128$.}\label{fig:a}
\end{figure}

\begin{example}\label{Exam:2}
The following fractional diffusion problem
  \begin{equation}\label{eq:5.3}
    \frac{\partial u(x,y,t)}{\partial t}={_0}D_x^{\alpha}u(x,y,t)+{_x}D_1^{\alpha}u(x,y,t)
        +{_0}D_y^{\beta}u(x,y,t)+{_y}D_{1}^{\beta}u(x,y,t)+f(x,y,t)
  \end{equation}
  is considered in the domain $\Omega=(0,1)^2$ and $t>0$ with the boundary conditions
  \begin{equation*}
    u(x,y,t)=0, \quad (x,y)\in \partial\Omega,\quad t\in [0, 1],
  \end{equation*}
  and initial value
  \begin{equation*}
    u(x,y,0)=u(x,y,0)=x^3(1-x)^3y^3(1-y)^3, \quad (x,y)\in [0, 1]^2.
  \end{equation*}
  The source term is
  \begin{equation*}
     \begin{split}
     f(x,y,t)=-\mathrm{e}^{-t}\Big[& x^3(1-x)^3y^3(1-y)^3\\
             &+\Big(\frac{\Gamma(4)}{\Gamma(4-\alpha)}\big(x^{3-\alpha}+(1-x)^{3-\alpha}\big)
              -\frac{3\Gamma(5)}{\Gamma(5-\alpha)}\big(x^{4-\alpha}+(1-x)^{4-\alpha}\big)\\
             &+\frac{3\Gamma(6)}{\Gamma(6-\alpha)}\big(x^{5-\alpha}+(1-x)^{5-\alpha}\big)
              -\frac{\Gamma(7)}{\Gamma(7-\alpha)}\big(x^{6-\alpha}+(1-x)^{6-\alpha}\big)\Big)y^3(1-y)^3\\
             &+\Big(\frac{\Gamma(4)}{\Gamma(4-\beta)}\big(y^{3-\beta}+(1-y)^{3-\beta}\big)
              -\frac{3\Gamma(5)}{\Gamma(5-\beta)}\big(y^{4-\beta}+(1-y)^{4-\beta}\big)\\
             &+\frac{3\Gamma(6)}{\Gamma(6-\beta)}\big(y^{5-\beta}+(1-y)^{5-\beta}\big)
              -\frac{\Gamma(7)}{\Gamma(7-\beta)}\big(y^{6-\beta}+(1-y)^{6-\beta}\big)\Big)x^3(1-x)^3\Big].
    \end{split}
  \end{equation*}
  And the exact solution of \ref{eq:5.2} is given by $u(x,t)=\mathrm{e}^{-t}x^3(1-x)^3y^3(1-y)^3$.
\end{example}
\begin{table}[h!]\fontsize{10pt}{12pt}\selectfont
  \begin{center}
  \caption{The maximum and $L^2$ errors and corresponding convergence rates to \ref{eq:5.2} approximated by the compact difference splitting schemes at $t=1$ with $\tau=h$.}\vspace{5pt}
  \begin{tabular*}{\linewidth}{@{\extracolsep{\fill}}*{1}{l}*{1}{r}*{8}{c}}
    \toprule
     & & \multicolumn{4}{c}{$(p,q)=(1,0), (\alpha, \beta)=(1.1, 1.7)$} & \multicolumn{4}{c}{$(p,q)=(1,-1),(\alpha, \beta)=(1.6, 1.9)$} \\
    \cline{3-6} \cline{7-10} \\[-8pt]
    Scheme & $N$ & $\|u^n-W^n\|_{\infty}$ & rate & $\|u^n-W^n\|$ & rate & $\|u^n-W^n\|_{\infty}$ & rate & $\|u^n-W^n\|$ & rate\\
    \toprule
    & 8 & 5.18337E-07 & - & 2.35962E-07 & - & 1.27051E-05 & - & 5.20284E-06 & - \\
    & 16 & 7.47852E-08 & 2.79  & 3.12736E-08 & 2.92 & 8.45409E-07 & 3.91  & 3.09460E-07 & 4.07 \\
    & 32 & 9.61857E-09 & 2.96  & 4.13976E-09 & 2.92 & 8.71092E-08 & 3.28  & 2.87346E-008 & 3.43 \\
    CLOD & 64 & 1.28508E-09 & 2.90  & 5.56504E-10 & 2.90 & 9.67026E-09 & 3.17  & 3.35453E-009 & 3.10 \\
    & 128 & 1.71211E-10 & 2.91  & 7.47594E-11 & 2.90 & 1.17863E-09 & 3.04  & 4.22753E-010 & 2.99 \\
    & 256 & 2.26863E-11 & 2.92  & 9.97266E-12 & 2.91 & 1.49079E-10 & 2.98  & 5.43715E-011 & 2.96 \\
     \midrule
    & 8 & 6.30467E-007 & - & 2.63036E-007 & - & 3.59941E-006 & - & 1.33801E-006 & - \\
    & 16 & 7.67884E-008 & 3.04  & 3.20134E-008 & 3.04 & 5.12851E-007 & 2.81  & 1.88287E-007 & 2.83 \\
    & 32 & 9.54417E-009 & 3.01  & 4.19273E-009 & 2.93 & 7.24724E-008 & 2.82  & 2.48958E-008 & 2.92 \\
    CPR& 64 & 1.28241E-009 & 2.90  & 5.60865E-010 & 2.90 & 9.12654E-009 & 2.99  & 3.25063E-009 & 2.94 \\
    & 128 & 1.71042E-010 & 2.91  & 7.51181E-011 & 2.90 & 1.15503E-009 & 2.98  & 4.22913E-010 & 2.94 \\
    & 256 & 2.26759E-011 & 2.92  & 1.00029E-011 & 2.91 & 1.49068E-010 & 2.95  & 5.48530E-011 & 2.95 \\
    \midrule
    & 8 & 6.28561E-007 & - & 2.54619E-007 & - & 3.01132E-006 & - & 1.15694E-006 & - \\
    & 16 & 7.67367E-008 & 3.03  & 3.12227E-008 & 3.03 & 4.37299E-007 & 2.78  & 1.67892E-007 & 2.78 \\
    & 32 & 9.54020E-009 & 3.01  & 4.12678E-009 & 2.92 & 6.67372E-008 & 2.71  & 2.30759E-008 & 2.86 \\
    CDouglas & 64 & 1.28220E-009 & 2.90  & 5.55710E-010 & 2.89 & 8.61113E-009 & 2.95  & 3.09798E-009 & 2.90 \\
    & 128 & 1.71028E-010 & 2.91  & 7.47144E-011 & 2.89 & 1.11175E-009 & 2.95  & 4.09795E-010 & 2.92 \\
    & 256 & 2.26748E-011 & 2.92  & 9.97005E-012 & 2.91 & 1.45187E-010 & 2.94  & 5.36563E-011 & 2.93 \\
    \midrule
   & 8 & 6.28561E-007 & - & 2.54619E-007 & - & 3.01132E-006 & - & 1.15694E-006 & - \\
    & 16 & 7.67367E-008 & 3.03  & 3.12227E-008 & 3.03 & 4.37299E-007 & 2.78  & 1.67892E-007 & 2.78 \\
    & 32 & 9.54020E-009 & 3.01  & 4.12678E-009 & 2.92 & 6.67372E-008 & 2.71  & 2.30759E-008 & 2.86 \\
   CD'yakonov & 64 & 1.28220E-009 & 2.90  & 5.55710E-010 & 2.89 & 8.61113E-009 & 2.95  & 3.09798E-009 & 2.90 \\
    & 128 & 1.71028E-010 & 2.91  & 7.47144E-011 & 2.89 & 1.11175E-009 & 2.95  & 4.09795E-010 & 2.92 \\
    & 256 & 2.26748E-011 & 2.92  & 9.97005E-012 & 2.91 & 1.45187E-010 & 2.94  & 5.36563E-011 & 2.93 \\
  \toprule
  \end{tabular*}\label{tab:2}
  \end{center}
\end{table}
In Table \ref{tab:2}, the errors $\|u^n-W^{n}\|,
\|u^n-W^{n}\|_\infty$ and their respective convergence rates are
presented for different uniformly space stepsizes, where
$W_{i,j}^{n}(\tau)=-\frac{1}{3}U_{i,j}^n(\tau)+\frac{4}{3}U_{i,j}^n(\tau/2)$
is the numerical solution by extrapolation in time, and $U_{i,j}^n$
satisfies the compact LOD scheme \ref{eq:4.15}, compact
Peaceman-Richardson scheme \ref{eq:4.16}, compact Douglas scheme
\ref{eq:4.17} and compact D'yakonov scheme \ref{eq:4.18},
respectively. The third order accuracy both in time and space is
verified, and in the computational process, the time costs are
largely reduced.
\section{Conclusion}
In \cite{Tian:12}, we introduce the weighted and shifted
Gr\"{u}nwald difference (WSGD) operators and show that the WSGD
operators have second order accuracy to approximate the fractional
derivatives. This paper is the sequel of \cite{Tian:12}. Based on
the WSGD operators, we further introduce the compact WSGD operators
(CWSGD) which have third order accuracy. Then we use the CWSGD
operators to establish the compact difference schemes for one and
two dimensional space fractional diffusion equations. And the
theoretical analysis of the stability and convergence of the schemes
is presented. The numerical results illustrate the effectiveness of
the compact difference approximation for the fractional problems and
confirm the convergence orders of the schemes.
\section{Acknowledgements}
The authors thank Prof Yujiang Wu for his constant encouragement and
support. This work was supported by the Program for New Century
Excellent Talents in University under Grant No. NCET-09-0438, the
National Natural Science Foundation of China under Grant No.
10801067, and the Fundamental Research Funds for the Central
Universities under Grant No. lzujbky-2010-63 and No.
lzujbky-2012-k26.



\addcontentsline{toc}{section}{References}
\begin{thebibliography}{99}
\bibitem{Barkai:02}E. Barkai, CTRW pathways to the fractional diffusion equation, Chem. Phys. {\bf 284} (2002) 13-27

\bibitem{Benson:00} D. A. Benson, S. W. Wheatcraft, M. M. Meerschaert, Application of a fractional advection-dispersion equation, Water Resour. Res. {\bf 36} (2000) 1403-1412

\bibitem{Chaves:98}A. S. Chaves, A fractional diffusion equation to describe L\'{e}vy flights, Phys. Lett. A. {\bf239} (1998) 13-16

\bibitem{Gorenflo:98} R. Gorenflo, F. Mainardi, Random walk models for space-fractional diffusion processes, Fract. Calc. Appl. Anal. {\bf 1} (1998) 167-191


\bibitem{Laub:05}A. J. Laub, Matrix Analysis for Scientists and Engineers, SIAM (2005)

\bibitem{Leveque:07}R. J. Leveque, Finite Difference Methods for Ordinary and Partial Differential Equations, SIAM (2007)

\bibitem{Liao:08}H. L. Liao, Z. Z. Sun, Maximum norm error bounds of ADI and compact ADI methods for solving parabolic equations,  Numer. Methods Partial Differential Equations. {\bf 26} (2008) 37-60

\bibitem{Li:11}C. Li, W. H. Deng, Y. J. Wu, Finite difference approximations and dynamics simulations for the L¨¦vy Fractional Klein-Kramers equation, Numer. Methods Partial Differential Equations. \href{http://onlinelibrary.wiley.com/doi/10.1002/num.20709/abstract}{DOI: 10.1002/num.20709}

\bibitem{Marchuk:1983}G. I. Marchuk, V. V. Shaidurov, Difference Methods and Their Extrapolations, Springer-Verlag, New York, 1983

\bibitem{Metzler:00} R. Metzler, J. Klafter, The random walk's guide to anomalous diffusion: A fractional dynamics approach, Phys. Rep. {\bf 339} (2000) 1-77

\bibitem{Metzler:04}R. Metzler, J. Klafter, The restaurant at the end of the random walk: recent developments in the description of anomalous transport by fractional dynamics, J. Phys. A: Math. Gen. {\bf 37} (2004) R161¨CR208

\bibitem{Meerschaert:04}M. M. Meerschaert, C. Tadjeran, Finite difference approximations for fractional advection-dispersion flow equations, J. Comput. Appl. Math. {\bf172} (2004) 65-77

\bibitem{Meerschaert:06}M. M. Meerschaert, C. Tadjeran, Finite difference approximations for two-sided space-fractional partial differential equations, Appl. Numer. Math. {\bf56} (2006) 80-90

\bibitem{Meerschaert:06b}M. M. Meerschaert, H. P. Scheffler, C. Tadjeran, Finite difference methods for two-dimensional fractional dispersion equation, J. Comput. Phys. {\bf211} (2006) 249-261

\bibitem{Podlubny:99}I. Podlubny, Fractional Differential Equations, Academic Press, San Diego (1999)

\bibitem{Minc:64}M. Marcus, H. Minc, A Survey of Matrix Theory and Matrix Inequalities, Allyn \& Bacon. Inc (1964)

\bibitem{Qin:11} J. Qin, T. Wang, A compact locally one-dimensional finite difference method for nonhomogeneous parabolic differential equations, Int. J. Numer. Meth. Biomed. Engng. {\bf 27} (2011) 128-142

\bibitem{Quarteroni:97}A. Quarteroni, A. Valli, Numerical Approximation of Partial Differential Equations, Springer (1997)

\bibitem{Scalas:00} E. Scalas, R. Gorenflo, F. Mainardi, Fractional calculus and continuous-time finance, Phys. A. {\bf 284} (2000) 376-384

\bibitem{Sengupta:04} T. K. Sengupta, G. Ganerwal, A. Dipankar, High accuracy compact schemes and Gibb's pheonomenon, J. Sci. Comput. {\bf 21} (2004) 253-268


\bibitem{Tadjeran:06}C. Tadjeran, M. M. Meerschaert, H. P. Scheffler,  A second-order accurate numerical approximation for the fractional diffusion equation, J. Comput. Phys. {\bf213} (2006) 205-213

\bibitem{Tadjeran:07}C. Tadjeran, M. M. Meerschaert, A second-order accurate numerical approximation for the two-dimensional fractional diffusion equation, J. Comput. Phys. {\bf220} (2007) 813-823

\bibitem{Tian:07} Z. F. Tian, Y. B. Ge, A fourth-order compact ADI method for solving two-dimensional unsteady convection-diffusion problems, J. Comput. Appl. Math. {\bf198} (2007) 268-286

\bibitem{Tian:12} W. Y. Tian, H. Zhou, W. H. Deng, A class of second order difference approximation for solving space fractional diffusion equations, submitted. \href{http://arxiv.org/abs/1201.5949}{arXiv:1201.5949 [math.NA]}

\bibitem{Zaslavsky:02}G. M. Zaslavsky, Chaos, fractional kinetic, and anomalous transport, Phys. Rep. {\bf 371} (2002) 461-580

\bibitem{Zhang:11} F. Z. Zhang, Matrix Theory: Basic Results and Techniques, 2nd ed., Springer (2011)

\end{thebibliography}
\end{document}